\documentclass[final]{colt2020} 

\usepackage{mathtools}
\usepackage{microtype}      
\usepackage{color}
\usepackage{wrapfig}

\newcommand{\RR}{\mathbb{R}}
\newcommand{\NN}{\mathbb{N}}

\renewcommand{\SS}{\mathbb{S}}

\newcommand{\Cc}{\mathcal{C}}

\newcommand{\Ff}{\mathcal{F}}

\newcommand{\Mm}{\mathcal{M}}

\newcommand{\Pp}{\mathcal{P}}
\newcommand{\Ss}{\mathcal{S}}

\newcommand{\Xx}{\mathcal{X}}
\newcommand{\Yy}{\mathcal{Y}}

\renewcommand{\d}{\, \mathrm{d}}

\newcommand{\id}{\mathrm{id}}

\newcommand{\qandq}{\quad \text{and} \quad}

\renewcommand{\div}{\mathrm{div}}

\DeclareMathOperator{\dist}{dist}
\DeclareMathOperator{\spt}{spt}

\graphicspath{ {./images/} }

\title[Implicit bias of wide two-layer neural networks]{Implicit Bias of Gradient Descent for Wide Two-layer Neural Networks Trained with the Logistic Loss}

\usepackage{times}

\coltauthor{%
 \Name{L\'ena\"ic Chizat} \Email{lenaic.chizat@universite-paris-saclay.fr}\\
 \addr Laboratoire de Math\'ematiques d'Orsay, CNRS, Universit\'e Paris-Saclay, France
 \AND
 \Name{Francis Bach} \Email{francis.bach@inria.fr}\\
 \addr INRIA, ENS, PSL Research University, Paris, France%
}

\begin{document}

\maketitle

\begin{abstract}%
  Neural networks trained to minimize the logistic (a.k.a.~cross-entropy) loss with gradient-based methods are observed to perform well in many supervised classification tasks. Towards understanding this phenomenon, we analyze the training and generalization behavior of infinitely wide two-layer neural networks with homogeneous activations. We show that the limits of the gradient flow on exponentially tailed losses can be fully characterized as a max-margin classifier in a certain non-Hilbertian space of functions. In presence of hidden low-dimensional structures, the resulting margin is independent of the ambiant dimension, which leads to strong generalization bounds. In contrast, training only the output layer implicitly solves a kernel support vector machine, which a priori does not enjoy such an adaptivity. Our analysis of training is non-quantitative in terms of running time but we prove computational guarantees in simplified settings by showing equivalences with online mirror descent. Finally, numerical experiments suggest that our analysis describes well the practical behavior of two-layer neural networks with ReLU activations and confirm the statistical benefits of this implicit bias.
\end{abstract}

\section{Introduction}\label{sec:intro}
Artificial neural networks are successfully used in a variety of difficult supervised classification tasks, but the mechanisms behind their performance remain unclear. The situation is particularly intriguing when the number of parameters of these models exceeds by far the number of input data points and they are trained with gradient-based methods until zero training error, without any explicit regularization. In this case, the training algorithm induces an \emph{implicit bias}: among the many classifiers which overfit on the training set, it selects a specific one which often turns out to perform well on the test set.
In this paper, we study the implicit bias of wide neural networks with two layers (i.e., with a single hidden-layer) trained with gradient descent on the logistic loss, or any loss with an exponential tail. Our analysis lies at the intersection of two lines of research that study (i) the implicit bias of gradient methods, and (ii) the training dynamics of wide neural networks.

\paragraph{Implicit bias of gradient methods.} \citet{soudry2018implicit} show that for linearly separable data, training a linear classifier with gradient descent on the logistic loss, or any loss with an exponential tail, implicitly leads to a max-margin linear classifier for the $\ell_2$-norm. This result together with results in the boosting literature~\citep{telgarsky2013margins} have led to a fruitful line of research. Fine analyses of convergence rates have been carried out by~\cite{nacson2018convergence, ji2019refined,ji2018risk}, and extensions to other gradient-based algorithms and to factored parameterizations are considered by~\cite{gunasekar2018characterizing}. Linear neural networks have been studied by~\cite{gunasekar2018implicit, ji2018gradient, nacson2019lexicographic}, and some properties in general non-convex cases are given by~\cite{xu2018will}. Closer to the present paper,  \cite{lyu2019gradient} show that for homogeneous neural networks the training trajectory converges in direction to a critical point of some nonconvex max-margin problem. In the present work, we improve this result for the two-layer case: we characterize the learnt classifier as the solution of a \emph{convex} max-margin problem. Importantly, this characterization is precise enough to enable a statistical analysis (see Section~\ref{sec:generalization}).

\paragraph{Dynamics of infinitely-wide neural networks.}  This fine characterization is made possible by looking at the infinite width limit of two-layer neural networks. This strategy has been used in several works to obtain insights on their statistical properties~\citep{bengio2006convex,bach2017breaking} or training behavior~\citep{nitanda2017stochastic, rotskoff2018neural, chizat2018global, mei2018mean, sirignano2018mean}, which can be described by a Wasserstein gradient flow~\citep{ambrosio2008gradient}.  In particular, \citet{chizat2018global} show that if the loss is convex, if the initialization is ``diverse enough'', and if the gradient flow of the objective converges, then its limit is a global minimizer. This result does not apply in our context because the gradient flow diverges, which turns out to be beneficial for the analysis of the implicit bias that we propose.

A general drawback of those \emph{mean-field} analyses is that they are mostly non-quantitative, both in terms of number of neurons and number of iterations. While some works have shown quantitative results by modifying the dynamics~\citep{mei2019mean,wei2019regularization, chizat2019sparse}, we do not take this path in order to stay close to the way neural networks are used in practice and because our numerical experiments suggest that those modifications are not necessary to obtain a good practical behavior.  
Finally, we stress that our analysis does not take place in the \emph{lazy training} regime~\citep{chizat2019lazy} which consists of training dynamics that can be analyzed in a perturbative regime around the initialization~\cite[see, e.g.,][]{li2018learning, jacot2018neural,du2018gradient}. Lazy training is another kind of implicit bias that amounts to training a linear model and does not lead to adaptivity results as those shown in Section~\ref{sec:generalization} (see Figure~\ref{fig:lazytogreedy} for an illustration in our context).

\subsection{Organization and contributions} 
After preliminaries on wide neural networks in Section~\ref{sec:infinitelywidenetworks}, we make the following contributions :
\begin{itemize}
\item In Section~\ref{sec:implicitbias}, we show that for a class of two-layer neural networks and for losses with an exponential tail, the classifier learnt by the non-convex gradient flow is a max-margin classifier for a certain functional norm known as the variation norm. 
\item When fixing the ``directions'' of the neurons (Section~\ref{sec:rate}), or when only training the output layer (Section~\ref{sec:convex}), we show that the dynamics implicitly performs \emph{online mirror ascent} on a sequence of smooth-margin objectives and thus naturally maximizes the margin. This leads to convergence guarantees in $O(\log(t)/\sqrt{t})$ in situations where no rate was previously known. 
\item In Section~\ref{sec:generalization}, we study the margins of those classifiers and prove dimension-independent generalization bounds for classification in presence of hidden linear structures.
\item We perform numerical experiments in Section~\ref{sec:numerics} for two-layer ReLU neural networks which confirm the statistical efficiency of this implicit bias in a high-dimensional setting.
\end{itemize}

In summary, we show that training two-layer ReLU neural networks implicitly solves a problem with strong statistical benefits. We stress however that the runtime of the algorithm is still unknown.

\subsection{Notation}

We denote by $\Mm(\RR^p)$ (resp.~$\Mm_+(\RR^p)$) the set of signed (resp.~nonnegative) finite Borel measures on~$\RR^p$ and by $\Pp(\RR^p)$ (resp.~$\Pp_2(\RR^p)$) the set of probability measures (resp.~with finite second moment). The set $\Delta^{m-1}=\{ p \in \RR^m_+\;;\; \mathbf{1}^\top p =1 \}$ is the simplex.

\section{Preliminaries on infinitely wide two-layer networks}\label{sec:infinitelywidenetworks}

\subsection{$2$-homogeneous neural networks}

We consider a binary classification problem with a training set $(x_i,y_i)_{i\in [n]}$ of $n$ pairs of observations with $x_i\in \RR^d$ and $y_i \in\{-1,+1\}$ and prediction functions of the form 
\begin{equation}\label{eq:finitewidth}
h_m(\mathbf{w},x)= \frac1m \sum_{j=1}^m \phi(w_j,x),
\end{equation}
 where $m\geq 1$ is the number of units and $\mathbf{w} = (w_j)_{j\in [m]} \in (\RR^p)^m$ are the trainable parameters. This setting covers two-layer neural networks where $m$ is the size of the hidden layer.  In this paper, we are interested in the over-parameterized regime where $m$ is large, and the prefactor $1/m$ is needed to obtain a non-degenerate limit. We refer to $\phi$ as a \emph{feature function}, and we focus on the case where $\phi$ is $2$-\emph{homogeneous} and \emph{balanced}:
 \begin{enumerate}
\item[{\sf (A1)}] The function $\phi$ is (positively) \emph{$2$-homogeneous} in its first variable, i.e.,  $\phi(rw,x)=r^2\phi(w,x)$ for all $(r,w,x)\in \RR_+ \times \RR^p  \times \RR^d$ and it is \emph{balanced}, which means that there is a map $T:\SS^{p-1}\to \SS^{p-1}$ such that for all $\theta\in \SS^{p-1}$, $\phi(T(\theta),\cdot)=-\phi(\theta,\cdot)$.
\end{enumerate}
Here are examples of models which satisfy  {\sf (A1)}:
 \begin{itemize}
 \item \emph{ReLU networks.} A two-layer neural network with the rectified linear unit (ReLU) activation function is obtained by setting $\phi(w,x)=b (a^\top (x,1) )_+$ where $w=(a,b)\in \RR^{d+2}$. This is the motivating example of this article. It is not differentiable but is covered by Theorem~\ref{th:bias_relu}.
 \item \emph{S-ReLU networks.}  With the function $\phi(w,x)=\epsilon \, (a^\top (x,1))_+^2$ where $w=(a,\epsilon)\in \RR^{d+1}\times \{-1,1\}$, we recover the same hypothesis class than two-layer neural networks with squared ReLU activation. This function is differentiable and rigorously covered by all theorems\footnote{Our arguments can indeed be applied to any situation where the parameter space can be factored as $\RR_+\times \Theta$ where $\Theta$ is a compact Riemannian manifold without boundary, see~\cite{chizat2019sparse}. For clarity, we limit ourselves to a parameter space $\RR^p$ (which corresponds to $\Theta=\SS^{p-1}$) while for S-ReLU, this would correspond to $\Theta = \SS^{p-1} \times \{-1,1\}$.}. 
 \end{itemize} 
 
\subsection{Parameterizing with a measure}\label{sec:measure}
The particular structure of two-layer neural networks allows for an alternative description of the predictor function. For $\mu\in \Pp_2(\RR^p)$, we define
\begin{equation}\label{eq:infinitewidth}
h(\mu,x) = \int_{\RR^p} \phi(w,x)\d\mu(w).
\end{equation}
Finite width networks as in Eq.~\eqref{eq:finitewidth} are recovered when $\mu$ is a discrete measure with $m$ atoms. 

The representation in Eq.~\eqref{eq:infinitewidth} can be reduced to a \emph{convex neural network} parameterized by an unnormalized measure~\citep{bengio2006convex}, as follows. We define the $2$-homogenous projection operator, $\Pi_2:\Pp_2(\RR^p)\to \Mm_+(\SS^{p-1})$ characterized by the property that, for any $\varphi\in \Cc(\SS^{p-1})$, it holds
\begin{equation}\label{eq:homogeneousprojection}
\int_{\SS^{p-1}} \varphi(\theta) \d[\Pi_2(\mu)](\theta) = \int_{\RR^p} \Vert w\Vert^2 \varphi(w/\Vert w\Vert) \d\mu(w),
\end{equation}
where the last integrand is extended by continuity at $w=0$.
This operator projects the mass of~$\mu$ on the unit sphere after re-weighting it by the squared distance to the origin. Seing $\Pi_2(\mu)$ as a measure on $\RR^p$ supported on the sphere, it holds by construction $h(\mu,\cdot) = h(\Pi_2(\mu),\cdot)$ for all $\mu \in \Pp_2(\RR^p)$. Note that the restriction to \emph{nonnegative} measures, which is not present in convex neural networks, does not change the expressivity of the model thanks to the assumption that $\phi$ is balanced in {\sf (A1)}.

\subsection{Max-margins and functional norms}\label{sec:functionalanalysis}
Given the training set  $(x_i,y_i)_{i\in [n]}$, the margin of a predictor $f:\RR^d\to \RR$ is given by 
$
\min_{i\in [n]} y_i f(x_i).
$
When the margin is strictly positive, the predictor makes no error on the training set and its value is typically seen as the worst confidence of the predictor.  Max-margin predictors are those that maximize the margin over a certain set of functions. When this set of functions is given by a unit ball for a certain norm $\Vert \cdot\Vert$, they solve
\[
\max_{\Vert f\Vert \leq 1} \min_{i\in [n]} y_i f(x_i).
\]
In this paper we deal with two notions of norms, that in turn define two types of max-margin classifiers. We refer to~\cite{bach2017breaking} for a more detailed presentation.

\paragraph{Variation norm.}  Given a feature function $\phi$ satisfying~{\sf (A1)}, we consider the space $\Ff_1$ of functions that can be written as 
$
f(x) = \int_{\SS^{p-1}} \phi(\theta,x) \d\nu(\theta),
$
where $\nu\in \Mm_+(\SS^{p-1})$ has finite mass (since $\phi$ is balanced, we could equivalently take $\nu \in \Mm(\SS^{p-1})$). The infimum of $\nu(\SS^{p-1})$ over all such decompositions defines a norm $\Vert f\Vert_{\Ff_1}$, sometimes called the \emph{variation norm} on $\Ff_1$~\citep{kurkova2001bounds}. The $\Ff_1$-max-margin of the training set is denoted $\gamma_1$ and given by 
\begin{equation}\label{eq:maxmargin1}
\gamma_1 \coloneqq \max_{\Vert f\Vert_{\Ff_1}\leq 1}\min_{i\in [n]}\  y_i f(x_i) = \max_{\substack{\nu \in \Mm_+(\SS^{p-1})\\ \nu(\SS^{p-1})\leq 1}}\min_{i\in [n]} \ y_i \int_{\SS^{p-1}}  \phi(\theta,x)\d\nu(\theta).
\end{equation}
For ReLU networks, the variation norm defined above does not \emph{a priori} coincide with the variation norm as defined by~\citet{bengio2006convex} and \citet{bach2017breaking} where the feature function is instead $\tilde \phi(\tilde \theta,x) = (a\cdot x + b)_+$ for $\tilde \theta =(a,b)\in \RR^{d}\times \RR$. Still, using $\phi$ or $\tilde \phi$ leads to norms which are equal up to a factor $2$ (see~\citet{neyshabur2014search} or Appendix~\ref{app:equivalencenorms}). See~\citet{savarese2019infinite, ongie2019function} for analytical descriptions of the space $\Ff_1$ for ReLU networks. 

\paragraph{RKHS norm.} Considering more specifically a two-layer neural network with activation function $\sigma:\RR\to \RR$, we can define another norm and function space, which leads to a reproducing kernel Hilbert space (RKHS). Let $\tau \in \Pp(\SS^{p-1})$ be the uniform measure on the sphere $\SS^{p-1}$ where $p=d+1$ and define $\Ff_2$ as the space of functions of the form
$
f(x) = \int_{\SS^{p-1}}  \sigma(b + c^\top x) g(b,c) \d\tau(b,c),
$
for some square-integrable function $g\in L^2(\tau)$. The infimum of $\Vert g\Vert_{L^2(\tau)} = \left(\int \vert g(b,c)\vert^2\d\tau(b,c)\right)^{\frac12}$ over such decompositions defines a norm $\Vert f\Vert_{\Ff_2}$. It is shown by~\citet{bach2017breaking} that $\Ff_2$ is a RKHS. The $\Ff_2$-max-margin of the training set is denoted $\gamma_2$ and given by
\begin{equation}\label{eq:maxmargin2}
\gamma_2 \coloneqq \max_{\Vert f\Vert_{\Ff_2}\leq 1}\min_{i\in [n]} \ y_i f(x_i) = \max_{\Vert g\Vert_{L^2(\tau)}\leq 1}\min_{i\in [n]} \ y_i \int_{\SS^{p-1}}  \sigma(b +c^\top x_i)g(b,c)\d\tau(b,c).
\end{equation}
This is a separable kernel support vector machine problem.

\paragraph{Statistical and computational properties.}  In Section~\ref{sec:generalization}, we will show that the margin $\gamma_1$ can be large even in high dimension when the dataset has hidden low dimensional structure, which leads to strong generalization guarantees, which is \emph{a priori} not true for $\gamma_2$. While $\Ff_2$-max-margin classifiers can be found with convex optimization techniques (such as training only the output layer, as shown in Section~\ref{sec:convex}), it is not clear \emph{a priori} how to find $\Ff_1$-max-margin classifiers. In the next section, we show that training an over-parameterized two-layer neural network precisely does that.

\subsection{Training dynamics in the infinite width limit}
\paragraph{Assumptions.}
Given a loss function $\ell:\RR\to \RR_+$, we define the empirical risk associated to a predictor $h_m(\mathbf{w},\cdot)$ of the form Eq.~\eqref{eq:finitewidth} as $\frac1n \sum_{i=1}^n \ell(- y_i h_m(\mathbf{w},x_i))$. Our analysis of the training dynamics relies on the following assumptions on the loss.
\begin{enumerate}
\item[{\sf (A2)}] The loss $\ell$ is differentiable with a locally Lipschitz-continuous gradient. It has an \emph{exponential tail} in the sense that $\ell(u)\sim \ell'(u) \sim \exp(u)$ as $u \to -\infty$, it is strictly increasing and there exists $c>0$ such that $\ell'(u)\geq c$ for $u\geq 0$. 
\end{enumerate}
The main examples are the logistic loss $\ell(u) = \log(1+\exp(u))$ and exponential loss $\ell(u) = \exp(u)$. Note that for our main result Theorem~\ref{th:biasWGF}, we do not assume convexity of the loss since only the tail behavior matters. We make the following assumptions on the feature function, in addition to~{\sf (A1)}.
\begin{enumerate}
\item[{\sf (A3)}] The family $(\phi(\cdot,x_i))_{i\in [n]}$ is linearly independent  and for $i\in [n]$, the function $\phi(\cdot,x_i)$ is differentiable with a Lipschitz-continuous gradient and subanalytic (i.e.,~its graph is locally the linear projection of a bounded semianalytic set).
\end{enumerate}
Let us comment these assumptions. Requiring that the family is linearly independent is equivalent to requiring that arbitrary labels can be fitted on the training input $(x_i)_{i\in [n]}$ within our hypothesis class $\{ x\mapsto \int \phi(\theta,x)\d\nu(\theta)\;;\; \nu \in \Mm_+(\SS^{p-1}) \}$. This assumption is satisfied by ReLU and S-ReLU networks~\citep{bach2017breaking} as soon as $x_i \neq x_{i'}$, $\forall i\neq i'$. The differentiability assumption is the most undesirable one because it excludes ReLU networks (but not S-ReLU networks). Although the training dynamic could potentially be defined without this assumption~\citep{lyu2019gradient}, the proof of Theorem~\ref{th:biasWGF} relies on it. Finally, subanalyticity is a mild assumption required in a technical proof step that invokes Sard's lemma. Functions defined by piecewise polynomials are subanalytic and thus both ReLU and S-ReLU networks satisfy it, see~\cite{bolte2006nonsmooth} for a definition.

\paragraph{Gradient flow of the smooth-margin objective.} In order to obtain simpler proofs we consider maximizing minus the logarithm of the empirical risk, instead of the empirical risk itself. This allows  to directly interpret the training dynamics as maximizing a \emph{smooth-margin} and leads to the same continuous time dynamics up to time reparameterization. We define the function $S:\RR^n\to \RR$ as
\begin{equation}\label{eq:smooth-margin}
S(u)= -  \log\bigg( \frac1n \sum_{i=1}^n \ell(- u_i)\bigg).
\end{equation}
When $\ell$ is the exponential, this function is known as the soft-min or the free energy and is concave \citep{mezard2009information}. Notice however that for now we do not assume that $\ell$ is the exponential but only {\sf (A2)}, in order to cover the case of the logistic function.

With a model of the form in Eq.~\eqref{eq:finitewidth}, this leads to an objective function $F_m:(\RR^p)^m\to \RR$ on the vector of parameters $\mathbf{w} = (w_j)_{j\in [m]}$ defined as 
\(
F_m(\mathbf{w}) = S(\hat h_m(\mathbf{w})),
\)
where we have denoted $\hat h_m(\mathbf{w}) = (y_i\, h_m(\mathbf{w},x_i))_{i\in [n]}$. We consider a (potentially random) initialization $\mathbf{w}(0) \in (\RR^p)^m$ and the (ascending) gradient flow of this objective function, which is a differentiable path $(\mathbf{w}(t))_{t\geq 0}$ starting from $\mathbf{w}(0)$ and such that for all $t\geq 0$,
\begin{equation}\label{eq:discreteGF}
\frac{d}{dt}\mathbf{w}(t) =  m \nabla F_m(\mathbf{w}(t)).
\end{equation}
 Up to the gradient sign, this gradient flow is an approximation of gradient descent~\citep{gautschi1997numerical, scieur2017integration} and stochastic gradient descent (SGD)~\cite[Thm. 2.1]{kushner2003stochastic} with small step sizes.\footnote{Although Theorem~\ref{th:biasWGF} below could be extended to discrete time analysis, this would be of little interest since the result is so far purely qualitative. In simpler settings, we study discrete time dynamics in Sections~\ref{sec:rate} and~\ref{sec:convex}.} Classical results guarantee that under Assumption {\sf (A2-3)}, this gradient flow is uniquely well defined.

\paragraph{Wasserstein gradient flow.}  Taking the point of view presented in Section~\ref{sec:measure}, we may interpret the training dynamics as a path
\(
\mu_{t,m} = \frac1m \sum_{j=1}^m \delta_{w_j(t)}
\)
in $\Pp_2(\RR^p)$. As we now explain, it turns out that this dynamics is a gradient flow for a function defined on $\Pp_2(\RR^p)$, which allows to seamlessly take the limit $m\to \infty$. Let $F$ be the functional on $\Pp_2(\RR^p)$ defined as
\[
F(\mu) = S(\hat h(\mu)),
\]
where similarly as above we define $\hat h:\Pp_2(\RR^p)\to \RR^n$ as $\hat h(\mu) = (y_i\, h(\mu,x_i))_{i\in [n]}$, and let $F'_{\mu}$ be its Fr\'echet derivative at $\mu$, which is represented by the function $F'_{\mu}(w)= \sum_{i=1}^n y_i\phi(w,x_i) \nabla_i S(\hat h(\mu))$. Let us give a definition of Wasserstein gradient flow (tailored to our smooth setting), which will be connected to the training dynamics of Eq.~\eqref{eq:discreteGF} in Theorem~\ref{th:WGF}.

\begin{definition}[Wasserstein gradient flow]\label{def:WGF}
A Wasserstein gradient flow for the functional $F$ is a path $(\mu_t)_{t\geq 0}$ such that there exists a flow $X : \RR_+\times \RR^p \to \RR^p$ satisfying $\mu_t = (X_t)_\# \mu_0$ (where $X_t(\cdot)=X(t,\cdot)$),  $X(0,\cdot)=X_0=\id_{\RR^p}$ and for all $(t,w)\in \RR_+\times \RR^p$,
\begin{equation}\label{eq:flowdef}
\frac{d}{dt} X(t,w) = \nabla F'_{\mu_t}(X(t,w)).
\end{equation}
\end{definition}
It can be directly checked that when $\mu_0$ is discrete, we recover the training dynamics defined in Eq.~\eqref{eq:discreteGF}. In this case, $w_j(t) = X(t,w_j(0))$ is the position (in parameter space) at time $t$ of the hidden unit initialized with parameters $w_j(0)$. The following theorem shows that Wasserstein gradient flows characterize the training dynamics of infinitely wide two-layer neural networks. It is an application of~\citet[Thm. 2.6]{chizat2018global}, see details in Appendix~\ref{app:meanfield} (hereafter, by convergence in $\Pp_2$, we mean weak convergence  and convergence of the second moments~\citep{ambrosio2008gradient}).

\begin{theorem}[Infinite width limit of training]\label{th:WGF}
Under {\sf (A1-3)}, if the sequence $(w_j(0))_{j \in \NN_*}$ is such that $\mu_{0,m}$ converges in $\Pp_2(\RR^p)$ to $\mu_0$, then $\mu_{t,m}$ converges in $\Pp_2(\RR^p)$ to the unique Wasserstein gradient flow of $F$ starting from $\mu_0$. The convergence is uniform on bounded time intervals.
\end{theorem}
This limit can be made quantitative using the geodesic convexity estimates of~\cite{chizat2018global}, and the stability results of~\citet[Thm.~11.2.1]{ambrosio2008gradient} but with an exponential dependency in time. In the different setting of the square loss, error estimates for SGD have been derived by~\cite{mei2018mean,mei2019mean}. This limit dynamics covers, but is not limited to, the lazy training dynamics studied by~\citet{li2018learning, jacot2018neural,du2018gradient} which here corresponds to a short time analysis when the initialization has a large variance (see Figure~\ref{fig:lazytogreedy} in Section~\ref{sec:numerics}).

\section{Main result: implicit bias of gradient flow}\label{sec:implicitbias}
We are now in position to state the main theorem of this paper, which characterizes the implicit bias of training infinitely wide two-layer neural networks with a loss with an exponential tail.
\begin{theorem}[Implicit bias]\label{th:biasWGF}
Under {\sf (A1-3)}, assume that $\Pi_2(\mu_0)$ has full support on $\SS^{p-1}$. If $\nabla S(\hat h(\mu_t))$ converges and $\bar \nu_t = \Pi_2(\mu_t)/([\Pi_2(\mu_t)](\SS^{p-1}))$ converges weakly to some $\bar \nu_\infty$, then this limit $\bar \nu_{\infty}$ is a maximizer for the $\Ff_1$-max-margin problem in Eq.~\eqref{eq:maxmargin1}.
\end{theorem}

We can make the following observations:
\begin{itemize}
\item The strength of this result is that the limit $\bar \nu_\infty$ of a~\emph{non-convex} dynamics is a \emph{global} minimizer of Eq.~\eqref{eq:maxmargin1}. Its proof relies, among other things, on a compatibility between the optimality conditions and the gradient flow dynamics, which is specific to the $2$-homogeneous case. 
\item It is an open question to prove that $\nabla S(\hat h(\mu_t))$ and $\bar \nu_t$ converge is this setting. Note that the unnormalized measure $\nu_t$ does \emph{not} converge, so the global convergence result from~\citet{chizat2018global} (which has a similar assumption regarding the existence of a limit) does not apply.
\item Unlike in the convex case~\citep{soudry2018implicit}, the dynamics does not completely forget where it started from. For instance, when initialized with a Dirac measure, the Wasserstein gradient flow can only converge to a Dirac measure, which is typically not a global minimizer.
\end{itemize}

Together, Theorems~\ref{th:WGF} and~\ref{th:biasWGF} give   asymptotic guarantees for training finite width neural networks.
\begin{corollary}\label{cor:interchange}
Under the assumptions of Theorem~\ref{th:biasWGF}, assume that the sequence $(w_j(0))_{j \in \NN_*}$ is such that $\mu_{0,m}$ converges in $\Pp_2(\RR^p)$ to $\mu_0$. Then, denoting $\bar \nu_{m,t} = \Pi_2(\mu_{m,t})/[\Pi_2(\mu_{m,t})](\SS^{p-1})$, it holds
\[
\lim_{m,t\to \infty} \left(\min_{i\in [n]} y_i \int \phi(\theta,x_i)\d\bar \nu_{m,t}\right) = \gamma_1.
\]
\end{corollary}
Typically, the sequence of initial neurons parameters $(w_j(0))_{j\in \NN_*}$ are sampled from a measure $\mu_0$ that satisfies the support condition of Theorem~\ref{th:biasWGF}, such as a Gaussian distribution. Note that limits in $t$ and $m$ can be interchanged so the convergence is not conditioned on a particular scaling.

\paragraph{Dealing with ReLU.} Let us now state a result which is a first step towards covering the case of ReLU networks, in spite of their non-differentiability\footnote{Added in revision. A similar idea for the initialization of ReLU was proposed independently in~\cite{wojtowytsch2020convergence}.}. Although its assumption (*) is arguably too strong, we  state it in order to point to technical open questions, see details in Appendix~\ref{sec:relucase}. For an input distribution $\rho \in \Pp(\RR^d)$  with a bounded density and bounded support $\Xx$, and $y:\Xx\to \{-1,1\}$ continuous, we consider the population objective
\begin{equation}\label{eq:objectiverelu}
F(\mu) = -\log \Big[ \int_\Xx \exp\Big(-y(x) h(\mu,x)\Big)\d\rho(x)\Big].
\end{equation}
\begin{theorem}\label{th:bias_relu}
There exists $(\mu_t)_{t\geq 0}$ a Wasserstein gradient flow of the objective Eq.~\eqref{eq:objectiverelu} with $\mu_0 = \mathcal{U}(\SS^{d})\otimes \mathcal{U}(\{-1,1\})$, i.e.,~input (resp.~output) weights uniformly distributed on the sphere (resp. on $\{-1,1\}$). If $\nabla S[\hat h(\mu_t)] $ converges weakly in $\Pp(\Xx)$, if $\bar \nu_t = \Pi_2(\mu_t)/([\Pi_2(\mu_t)](\SS^{p-1}))$ converges weakly in $\Pp(\SS^{p-1})$ and if (*) $F'_{\mu_t}$ converges  in $\mathcal{C}^1_{loc}$ to some $F'$ that satisfies the Morse-Sard property (see details in Appendix~\ref{sec:relucase}), then $h(\bar \nu_{\infty},\cdot)$ is a maximizer for 
$
\max_{\Vert f\Vert_{\Ff_1}\leq 1} \min_{x\in \Xx} y(x)f(x).
$
\end{theorem}

\section{Insights on the convergence rate and choice of step-size}\label{sec:rate}
While making Corollary~\ref{cor:interchange} quantitative in terms of number of neurons and the number of iterations is left as an open question, it is of practical importance to better understand the effect of the choice of step-size. In this section, we look at a simplified dynamics where the direction of each parameter $w_j(t)$ is fixed after initialization and only its magnitude evolves. A complete discrete-time analysis is possible in this case, using tools from convex analysis. 

We consider a model of the form of Eq.~\eqref{eq:finitewidth} but with $w_j(t)$ written as $r_j(t)\theta_j$, where $r_j(t) \in \RR_+$ is trained and $\theta_j \in \SS^{p-1}$ is fixed at initialization. Plugging this model into the soft-min loss~\eqref{eq:smooth-margin} yields an objective function $F_m : \RR_+^m\to \RR$ defined as
$$
F_m(r) = -\log\bigg( \frac1n\sum_{i=1}^n \exp\bigg(- \frac1m \sum_{j=1}^m z_{i,j} r^2_j\bigg)\bigg),
$$
where $z_{i,j} = y_i\phi(\theta_j,x_i)$ are the signed fixed features. We focus on the exponential loss $\ell=\exp$ in this section and the next one for simplicity. We study the gradient ascent dynamics with initialization $r(0) \in \RR_+^m$ and sequence of step-sizes $(\eta(t))_{t\in \NN}$,
\[
r(t+1) = r(t) + \eta(t) m \nabla F_m(r(t)).
\]
This dynamics is studied by~\citet{gunasekar2018characterizing} where it is shown to converge to a max $\ell_1$-margin classifier without a rate. In the next proposition, we prove  convergence of the best iterate to maximizers at an asymptotic rate $\log(t)/\sqrt{t}$, by exploiting an analogy with online mirror ascent. 

\begin{proposition}\label{prop:OMA}
Let $a_j(t) = r_j(t)^2/m$ for $j\in [m]$, $\beta(t) = \Vert a(t)\Vert_1$ and $\bar a(t) = a(t)/\beta(t)$. For the step-sizes $\eta(t) = 1/(16\Vert z \Vert_\infty \sqrt{t+1})$ and a uniform initialization $r(0)\propto \mathbf{1}$, it holds
$$
\max_{0\leq s\leq t-1} \min_{i\in [n]} z_i^\top \bar a(s)  \geq \gamma_1^{(m)} - \frac{\Vert z\Vert_\infty}{\sqrt{t}}(8\log(m) + \log(t)+1) - \frac{4 B \log n}{\sqrt{t}} .
$$
where $\gamma_1^{(m)} := \max_{a \in \Delta^{m-1}}\min_{i \in [n]} z_i^\top a$ and $B:=\sum_{s=0}^{\infty} \frac{1}{\beta(s)\sqrt{s+1}}<\infty$ when $\gamma_1^{(m)}>0$.
\end{proposition}
In the proof of Lemma~\ref{lem:betaunbounded}, it can be seen that our bound on $B$ grows to $\infty$ as $\gamma_1^{(m)}$ goes to zero.

\paragraph{Proof idea.} To prove Proposition~\ref{prop:OMA}, we consider the family of \emph{smooth-margin} functions 
$$
G_\beta(a) =  -\frac1\beta \log\Big( \frac1n\sum_{i=1}^n \exp\big(-  \beta \sum_{j=1}^m z_{i,j} a_j\big)\Big),
$$
and we show that $\bar a(t)$ approximately follows \emph{online mirror ascent} for the sequence of concave functions $G_{\beta(t)}$ in the simplex $\Delta^{m-1}$ with step-sizes $\eta(t)$. It then only remains to apply classical bounds for mirror descent and use the fact that $\vert \min_{i\in [n]} z_i^\top a - G_\beta(a)\vert \leq \log (n) /\beta$.  This algorithm thus implicitly performs online optimization on the regularization path. It is also analogous to smoothing techniques in non-smooth optimization~\citep{nesterov2005smooth}.

\paragraph{Continuous limit.} Using the notations from Section~\ref{sec:infinitelywidenetworks}, the dynamics $\sum_{j=1}^m \bar a_j(t)\delta_{\theta_j}$  solves
$$
\gamma_1^{(m)} := \max_{\substack{\nu\in\Mm_+(\SS^{p-1})\\ \nu(\SS^{p-1})\leq 1}}  \min_{i\in [n]} y_i \int_{\SS^{p-1}} \phi(\theta,x_i)\d\nu(\theta) \quad\text{subject to}\quad \text{$\nu$ supported on  $\{\theta_j\}_{j\in [m]}$}.
$$
When $\frac1m \sum_{j=1}^m \delta_{\theta_j}$ converges to the uniform measure on the sphere, we thus recover the same implicit bias as in Theorem~\ref{th:biasWGF} and $\gamma_1^{(m)}\to \gamma_1$ (note that the logarithmic dependency in $m$ in Proposition~\ref{prop:OMA} could be removed with a slightly finer analysis as done in~\citet{chizat2019sparse}). While functions in $\Ff_1$ may be well-approximated with a small number of neurons~\citep{bach2017breaking,jones1992simple}, this is not anymore true if the positions $\{\theta_j\}_{j\in [m]}$ of those neurons are fixed a priori (see~\cite{barron1993universal} for exponential lower bounds in a similar setting). In Theorem~\ref{th:biasWGF}, positions are allowed to vary during training: this makes its setting more challenging but also much more relevant.

\section{Training only the output layer}\label{sec:convex}
For two-layer neural networks, it is instructive to compare the implicit bias of training both layers (as in Section~\ref{sec:implicitbias}) versus that of training only the output layer, the input layer being initialized randomly and fixed. This model gives the objective function $F : \RR^m\to \RR$ defined as
\[
F(r) = -\log \bigg( \frac1n \sum_{i=1}^n \exp\bigg( -\frac1m \sum_{j=1}^m z_{i,j} r_j \bigg) \bigg),
\] 
where $z_{i,j} = y_i \sigma(b_j + x_i^\top c_j)$ is the signed output of neuron $j$ for the training point $i$ and $\sigma:\RR\to\RR$ is the non-linearity, e.g., $\sigma(u)=\max\{0,u\}$ for ReLU networks. We study the gradient ascent dynamics with initialization $r(0)\in \RR^m$ and sequence of step-sizes $(\eta(t))_{t\in \NN}$:
$$
r(t+t) = r(t) +\eta(t) m \nabla F(r(t)).
$$
\citet{soudry2018implicit} show that for a step-size of order $1/\sqrt{t}$, this dynamics converges in $O(\log(t)/\sqrt{t})$ to a max $\ell_2$-margin classifier. Next, we show that it converges in $O(1/\sqrt{t})$ for larger, non-vanishing step-sizes, with a different proof technique. The fact that the algorithm converges at essentially the same speed for very different step-sizes shows an advantageous self-regularizing property.

\begin{proposition}\label{prop:PGD}
Let $a(t) = r(t)/m$, $\beta(t) = \max\{ 1, \max_{0\leq s\leq t} \sqrt{m}\Vert a(t)\Vert_2\}$ and $\bar a(t) = a(t)/\beta(t)$. Assume  $\gamma_2^{(m)} := \max_{\sqrt{m}\Vert a\Vert_2\leq 1} \min_{i\in [n]} z_i^\top a>0$. For the step-sizes $\eta(t) = \beta(t) {\sqrt{2}}/(\Vert z\Vert_\infty\sqrt{t+1})$ and initialization $r(0)=0$, it holds
$$
\max_{0\leq s\leq t-1}\min_{i\in [n]} z_i^\top \bar a(s) \geq \gamma_2^{(m)} -  \frac{\Vert z\Vert_\infty}{\sqrt{t}}\Big( 2\sqrt{2} +\frac{\sqrt{3}\log n}{\gamma_2^{(m)}} \Big) .
$$
\end{proposition}

\paragraph{Proof idea.} Similarly to the proof of Proposition~\ref{prop:OMA}, we show that $\bar a(t)$ follows an \emph{online projected gradient ascent} for the sequence of functions $G_{\beta(t)}$ in the ball $\{ a \in \RR^m \;;\; \Vert a\Vert_2 \leq 1/\sqrt{m}\}$ and with step-sizes $\eta(t)/(m\beta(t))$. From there, we use standard optimization results and prove that $\beta(t)\to \infty$ to conclude.
Note that a different reduction to mirror descent for this dynamics was also exhibited by~\citet{ji2019refined} and used to derive tight convergence rates but with a much smaller step-size than in Proposition~\ref{prop:PGD} and with a different Bregman divergence.

\paragraph{Random features for kernel max-margin classifier.} Using the notations from Section~\ref{sec:infinitelywidenetworks}, the dynamics $(r_j)_{j\in [m]}$ converges to a solution to
$$
\gamma_2^{(m)} = \max_{g \in L^2(\tau_m)} \min_{i\in [n]} y_i \int g(b,c) \sigma(b + x_i^\top c) \d\tau_m(b,c) \quad \text{subject to}\quad \Vert g\Vert_{L^2(d\tau_m)}\leq 1,
$$
where $\tau_m = \frac1m \sum_{j=1}^m \delta_{(b_j,c_j)}$. Typically, the input layer parameters are sampled from a distribution $\tau \in \RR^{1+d}$, which corresponds to a random feature approximation for the $\Ff_2$-max-margin problem of Eq.~\eqref{eq:maxmargin2} and we have $\gamma_2^{(m)}\to \gamma_2$. In stark contrast to the space $\Ff_1$, functions in $\Ff_2$ can be well approximated with few random features even in high dimension. See~\citet{rahimi2008random,bach2017equivalence} for a analysis of the number of features needed for an approximation with error~$\varepsilon$, typically of order $1/\varepsilon^2$.

\section{Dimension independent generalization bounds}\label{sec:generalization}
In this section, we give arguments showing the favorable statistical properties of the bias exhibited in Theorem~\ref{th:biasWGF} for ReLU networks. We propose to measure the complexity of the dataset $S_n = (x_i,y_i)_{i=1}^n$ with the following \emph{projected interclass distance} defined, for $r\in [d]$, as 
\begin{equation}\label{eq:interclass}
\Delta_r(S_n) :=  \sup_{P} \left\{ \inf_{y_i\neq y_{i'}}  \Vert P(x_i) - P(x_{i'})\Vert_2 \;;\;  \text{$P$ is a rank-$r$ orthogonal projection}\right\}.
\end{equation}
For each dimension $r$, it looks for the $r$-dimensional subspace which maximizes the distance between the two classes. Interclass distance often appears in the statistical analysis of classification problems~\citep[see, e.g.,][]{li2018learning} often complemented with ``clustered data'' assumptions. Our definition is designed to capture the fact that if $\Delta_r\approx \Delta_d$ for $r\ll d$, then there is a hidden structure which can be exploited for statistical efficiency.

\begin{theorem}[Generalization bound]\label{th:generalization}
For any $\epsilon \in (0,1)$ and $r\in [d]$, there exist $C(r),C_\epsilon(r)>0$ such that the following holds. If $(x,y)\sim \mathbb{P}$ is such that for some $R>0$ and $\Delta_r({\mathbb{P}})\leq C(r)$, it holds $\Delta_r(S_n) \leq \Delta_r(\mathbb{P})$ and $\Vert x\Vert_2\leq R$ almost surely, then it holds with probability at least $1-\delta$ over the choice of i.i.d.~samples $S_n = (x_i,y_i)_{i=1}^n$, for $f$ the $\Ff_1$-max-margin classifier on $S_n$,
\[
\mathbb{P}[yf(x)<0] \leq \frac{C_\epsilon(r)}{\sqrt{n}} \left(\frac{R}{\Delta_r(\mathbb{P})}\right)^{\frac{r+3}{2-\epsilon}} + \sqrt{\frac{\log(B)}{n}} + \sqrt{\frac{\log(1/\delta)}{2n}}
\]
where $B=\log_2(4(R+1)C_2(r))+(r+2)\log_2(R/\Delta_r(\mathbb{P}))$. The same bound applies to the $\Ff_2$-max-margin classifier for $r=d$.
\end{theorem}
\paragraph{Proof idea.} We first lower-bound the margins $\gamma_1$ and $\gamma_2$ in terms of $\Delta_r(S_n)$ and then apply margin-based generalization bounds~\citep{koltchinskii2002empirical} and bounds on the Rademacher complexity of the unit ball of $\Ff_1$ and $\Ff_2$.

The rate $n^{-1/2}$ is suboptimal (an exponential decay of the test error is possible in this context~\citep{pillaud2018exponential}), but this statement is a strong non-asymptotic bound:  for the $\Ff_1$-max-margin classifier, $d$ does not appear in the exponent of the ratio $R/\Delta_r(\mathbb{P})$, which characterizes the difficulty of the problem. Related generalization bounds are given by~\citet{wei2019regularization}, where a factor $d$ improvement for the $\Ff_1$ versus $\Ff_2$-max-margin classifier is shown on a specific example. Also~\citet{montanari2019generalization} prove generalization bounds for linear max-margin classifiers.

\begin{figure}
\centering
\begin{tabular}{rcc|ccc|ccc|cc}
\rotatebox{90}{$\quad$both layers}&\includegraphics[scale=0.18,trim=3.5cm 2cm 0.5cm 0,clip]{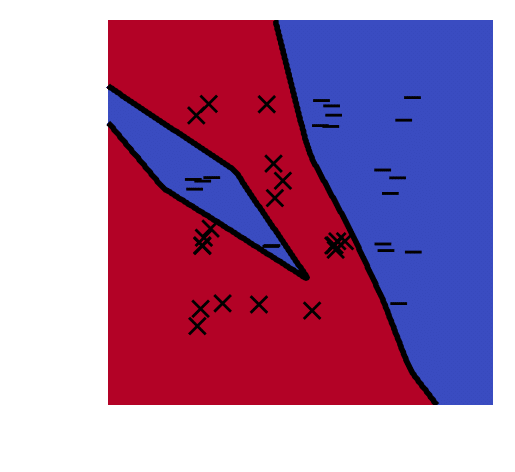} &&& \includegraphics[scale=0.18,trim=3.5cm 2cm 0.5cm 0,clip]{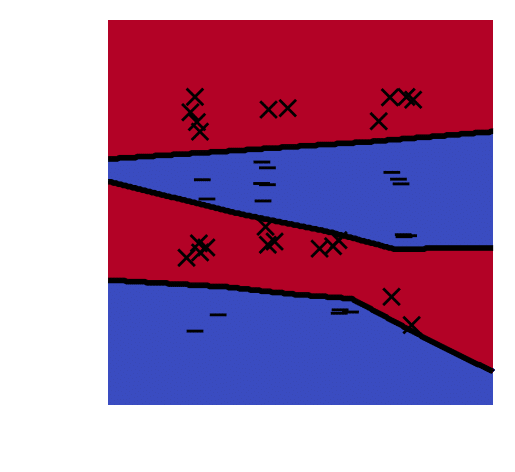}  &&& \includegraphics[scale=0.18,trim=3.5cm 2cm 0.5cm 0,clip]{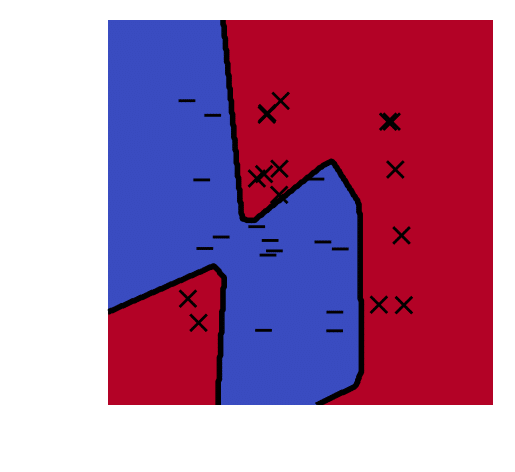}&&&  \includegraphics[scale=0.18,trim=3.5cm 2cm 0.5cm 0,clip]{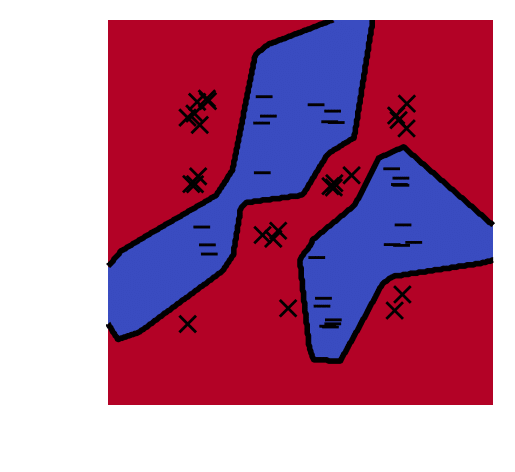} \\
\rotatebox{90}{$\;\;$output layer}&\includegraphics[scale=0.18,trim=3.5cm 2cm 0.5cm 0,clip]{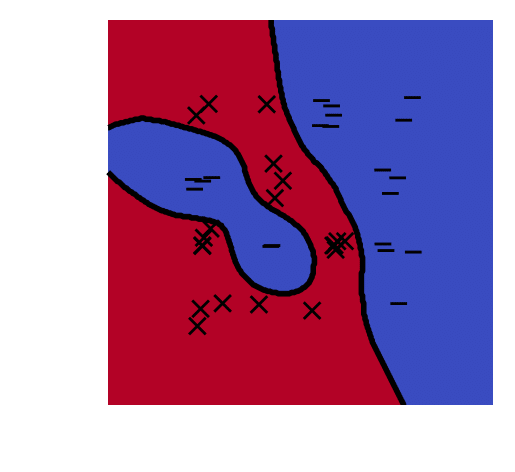} &&& \includegraphics[scale=0.18,trim=3.5cm 2cm 0.5cm 0,clip]{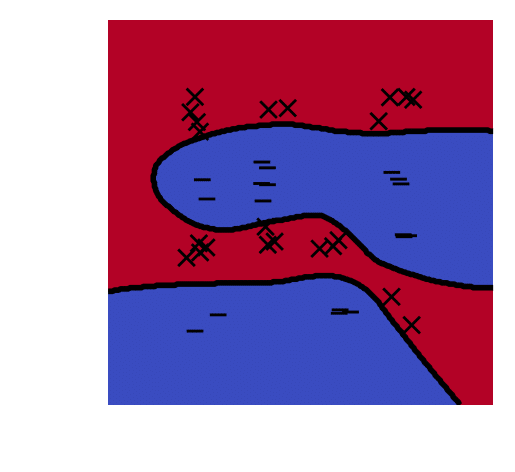}  &&& \includegraphics[scale=0.18,trim=3.5cm 2cm 0.5cm 0,clip]{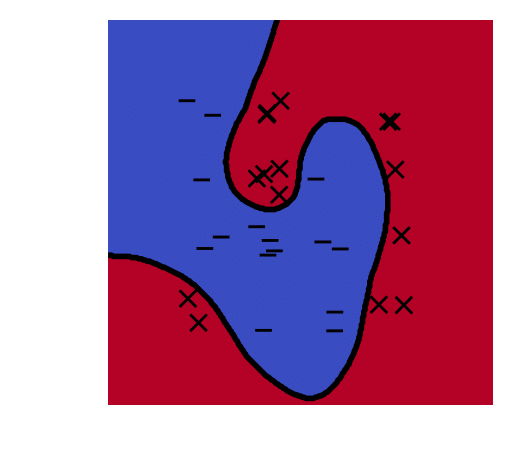}&&&  \includegraphics[scale=0.18,trim=3.5cm 2cm 0.5cm 0,clip]{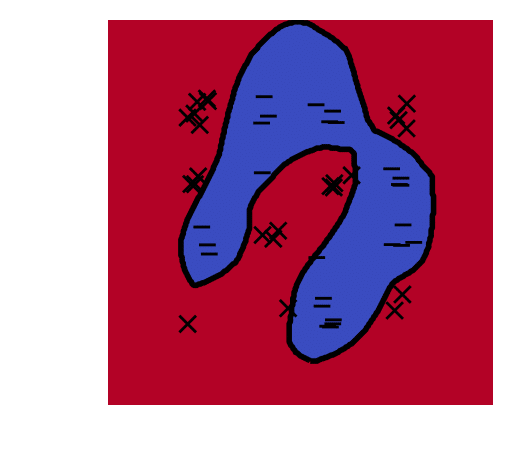}  \\
\end{tabular}
\vspace{-0.3cm}
\caption{Comparison of the implicit bias of training (top) both layers versus (bottom) the output layer  for ReLU networks with $d=2$ and for $4$ different random training sets.}\label{fig:illustration}
\end{figure}

\section{Numerical experiments}\label{sec:numerics}
In this section, we consider a large ReLU network with $m=1000$ hidden units, and compare the implicit bias and statistical performances of training both layers -- which leads to a max margin classifier in $\Ff_1$ -- versus the output layer -- which leads to max margin classifier in $\Ff_2$. The experiments are reproducible with the Julia code that can be found online\footnote{\url{https://github.com/lchizat/2020-implicit-bias-wide-2NN}}. 

\paragraph{Setting.}  Our data distribution is supported on $[-1/2, 1/2]^d$ and is generated as follows. In dimension $d=2$, the distribution of input variables is a mixture of $k^2$ uniform distributions on disks of radius $1/(3k-1)$ on a uniform $2$-dimensional grid with step $3/(3k-1)$, see Figure~\ref{fig:performance}(a) for an illustration with $k=3$. In dimension larger than $2$, all other coordinates follow a uniform distribution on $[-1/2,1/2]$. Each cluster is then randomly assigned a class in $\{-1,+1\}$. For such distributions, the parameters appearing in Theorem~\ref{th:generalization} satisfy $\Delta_2(\mathbb{P})\geq 1/(3k-1)$ and $R\leq \sqrt{d}$. 

\paragraph{Low dimensional illustrations.} Figure~\ref{fig:illustration} illustrates the differences in the implicit biases when $d=2$. It represents a sampled training set and the resulting decision boundary between the two classes for $4$ examples. The $\Ff_1$- max-margin classifier  is non-smooth and piecewise affine, which comes from the fact that the mass constraint in Eq.~\eqref{eq:maxmargin1} favors sparse solutions. In contrast, the max-margin classifier in $\Ff_2$ has a smooth decision boundary, which is typical of learning in a RKHS.

\begin{figure}\centering
\subfigure[Distribution][b]{\includegraphics[scale=0.43, trim = 0 -23 0 0, clip]{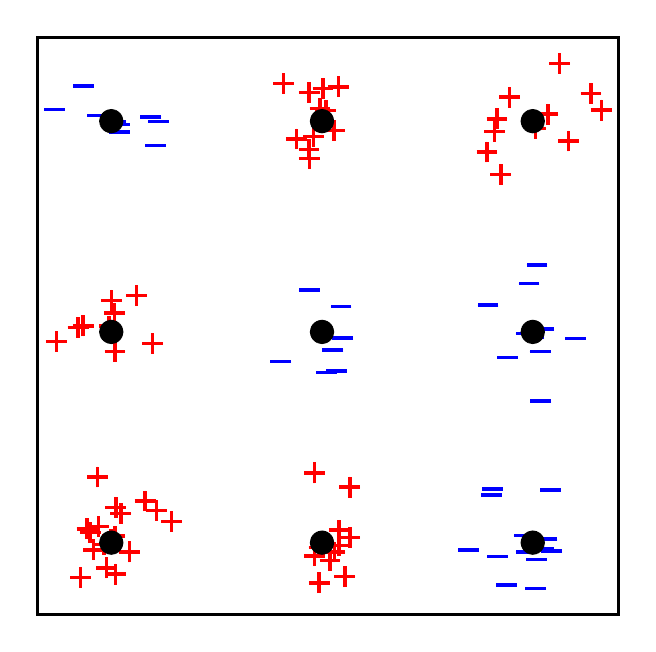}\label{fig:setting}}\hspace{0.01cm}
\subfigure[Test error vs.~$n$][b]{\includegraphics[scale=0.4]{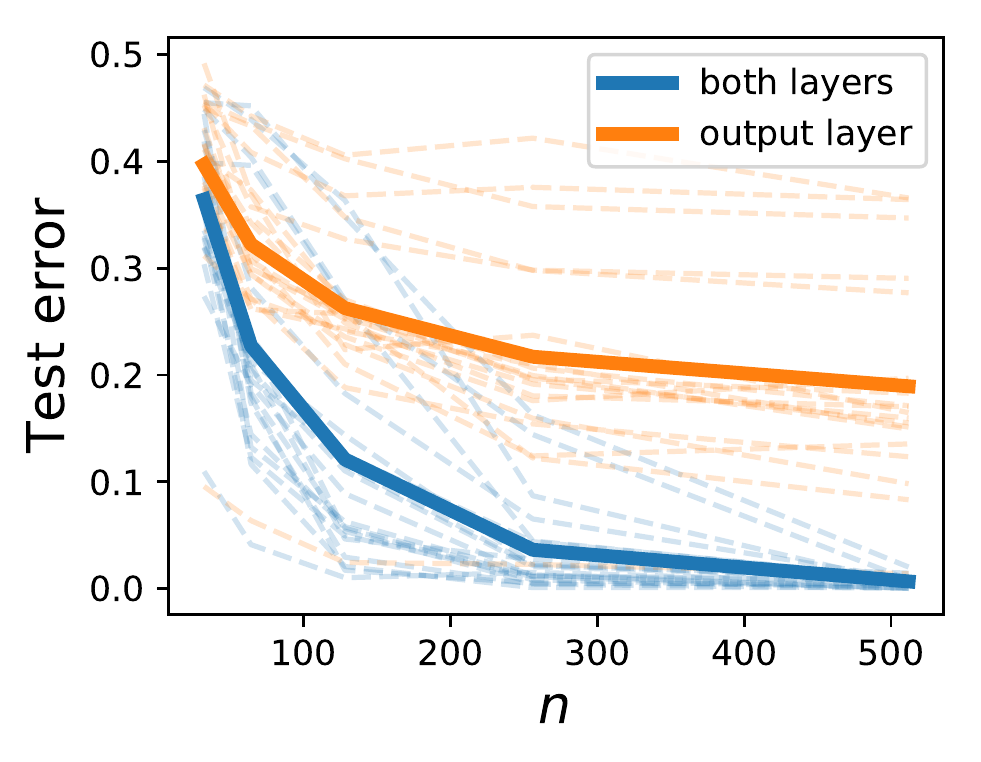}\label{fig:n}}\hspace{0.01cm}
\subfigure[Test error vs.~$d$][b]{\includegraphics[scale=0.4]{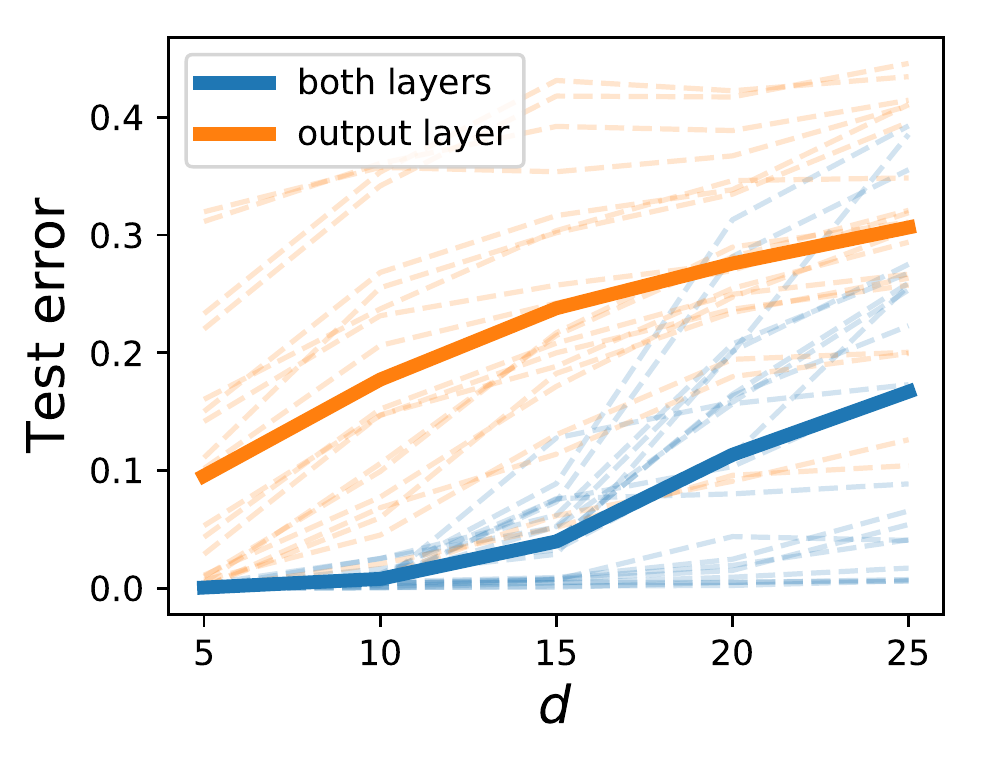}\label{fig:d}}\hspace{0.01cm}
\subfigure[Margin vs.~$m$][b]{\includegraphics[scale=0.4]{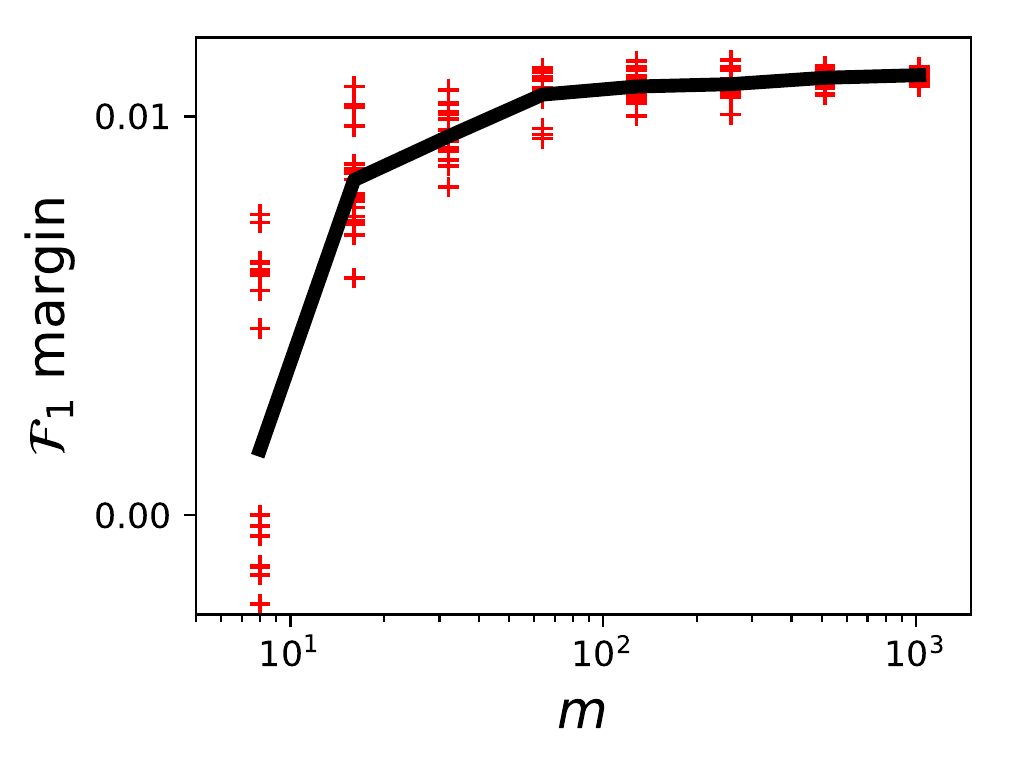}\label{fig:m}}
\caption{(a) Projection of the data distribution on the two first dimensions, (b) test error as a function of $n$ with $d=15$, (c) test error as a function of $d$ with $n=256$ (d) $\Ff_1$-margin at convergence as a function of $m$ when training both layers with $n=256$, $d=15$.}\label{fig:performance}
\end{figure}

\paragraph{Performance.} In higher dimensions, we observe the superiority of training both layers by plotting the test error versus $m$ or $d$ on Figure~\ref{fig:performance}(b) and~\ref{fig:performance}(c). We ran $20$ independent experiments with $k=3$ and show with a thick line the average of the test error $\mathbb{P}(yf(x)<0)$ after training. Note that $R$ grows as $\sqrt{d}$ so the dependency in $d$ observed in Figure~\ref{fig:performance}(c) is not in contradiction with Theorem~\ref{th:generalization}. Finally, Figure~\ref{fig:performance}(d) illustrates Corollary~\ref{cor:interchange} and shows the $\Ff_1$-margin after training both layers. For each $m$, we ran $30$ experiments using fresh random samples from the same data distribution.

\paragraph{Two implicit biases in one dynamics.} In Figure~\ref{fig:lazytogreedy}, we illustrate for $d=2$ a case where \emph{two} different kinds of implicit biases show up in a single dynamics ($t$ is the number of iterations with a constant step-size). We initialize the ReLU network with a large variance ($\mathcal{N}(0,40^2)$). The model is at first in the~\emph{lazy regime}~\citep{chizat2019lazy} and follows closely the dynamics of its linearization around initialization, which converges to the max-margin classifier for the \emph{tangent kernel}~\citep{jacot2018neural}. It then converges to the $\Ff_1$-max-margin classifier as suggested by Theorem~\ref{th:biasWGF}. In order to observe this intermediate implicit bias, one needs an initial step-size inversely proportional to the scale of the initialization~\citep{chizat2019lazy}.

\begin{figure}\centering
\subfigure[$t=10^2$ ][b]{\includegraphics[scale=0.28, trim = 25 10 0 0, clip]{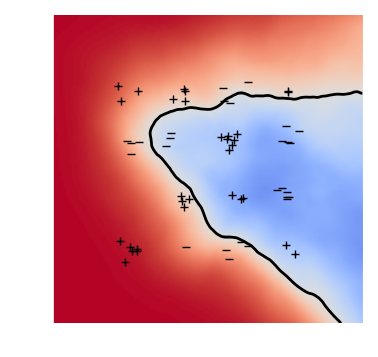}}\hspace{0.01cm}
\subfigure[$t=5*10^2$][b]{\includegraphics[scale=0.28,  trim = 25 10 0 0, clip]{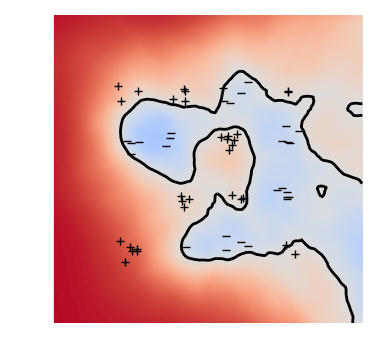}}\hspace{0.01cm}
\subfigure[$t=5*10^3$][b]{\includegraphics[scale=0.28,  trim = 25 10 0 0, clip]{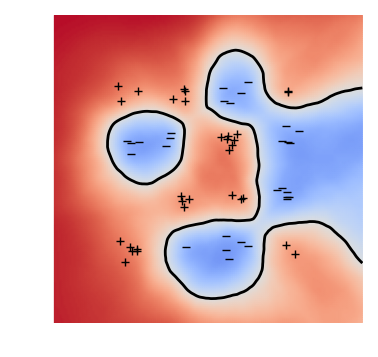}}\hspace{0.01cm}
\subfigure[$t=3*10^4$][b]{\includegraphics[scale=0.28,  trim = 25 10 0 0, clip]{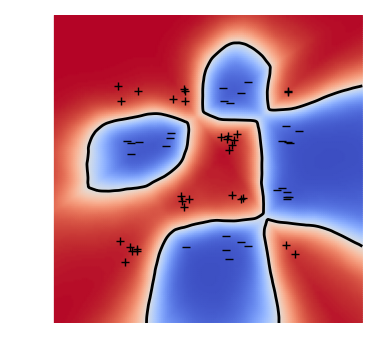}}\hspace{0.01cm}
\subfigure[$t=6*10^4$][b]{\includegraphics[scale=0.28,  trim = 25 10 0 0, clip]{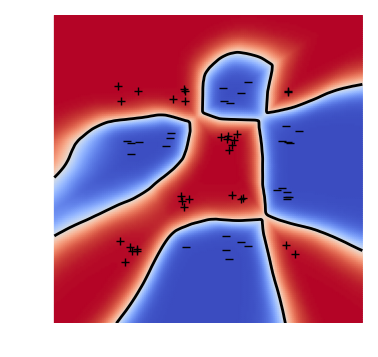}}\hspace{0.01cm}
\subfigure[$t=3.10^5$][b]{\includegraphics[scale=0.28,  trim = 25 10 0 0, clip]{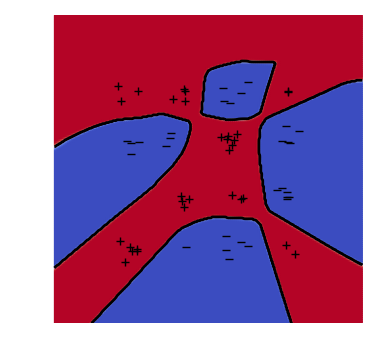}}\caption{Dynamics of the classifier while training both layers for an initialization with a large variance and a small initial step-size. The classifier first approaches the max-margin classifier for the \emph{tangent kernel}~\citep{jacot2018neural} (c) and eventually converges to the $\Ff_1$-max-margin (f).}\label{fig:lazytogreedy}
\end{figure}

\section{Conclusion}
We have shown that for wide two-layer ReLU neural networks, training both layers or only the output layer leads to very different implicit biases. When training both layers, the classifier converges to a max-margin classifier for a non-Hilbertian norm, which enjoys favorable statistical properties. Interestingly, this problem does not seem to be directly solvable with known convex methods in high dimension. Proving complexity guarantees for this non-convex gradient flow is an important open question for future work. In particular, even for   infinite width, continuous time dynamics as in Theorem~\ref{th:biasWGF}, it is still unknown whether a convergence rate can be given under reasonable conditions.

\newpage
\section*{Acknowledgements}
Part of this work was carried through while the first author was visiting the Chair of Statistical Field Theory at the \'Ecole Polytechnique F\'ed\'erale de Lausanne (EPFL), Switzerland. This work was funded in part by the French government under management of Agence Nationale de la Recherche as part of the ``Investissements d'avenir'' program, reference ANR-19-P3IA-0001 (PRAIRIE 3IA Institute). We also acknowledge 
 support the European Research Council (grant SEQUOIA 724063).

\bibliography{LC}
\newpage

\appendix

\section{Organization of the appendix}
\begin{itemize}
\item In Appendix~\ref{app:equivalencenorms}, we prove the equivalence between our definition of the variation norm in Section~\ref{sec:functionalanalysis} with the one that is used in the literature on convex neural networks.
\item In Appendix~\ref{app:meanfield}, we discuss properties of the Wasserstein gradient flow and justify Theorem~\ref{th:WGF}.
\item In Appendix~\ref{app:maintheorem}, we prove our main theorem Theorem~\ref{th:biasWGF} and its corollary.
\item In Appendix~\ref{app:ratesL1}, we prove Proposition~\ref{prop:OMA} on the convergence rate with fixed ``positions''.
\item In Appendix~\ref{app:output}, we prove Proposition~\ref{prop:PGD} on the convergence rate when training the output layer.
\item In Appendix~\ref{app:generalization}, we prove Theorem~\ref{th:generalization} on the margins and generalization performance.
\item In Appendiz~\ref{sec:relucase}, we prove Theorem~\ref{th:bias_relu} which covers the case of ReLU networks.
\end{itemize}

\section{Equivalence of two variation norms}\label{app:equivalencenorms}

Let us introduce, for ReLU networks, the variation norm introduced in Section~\ref{sec:functionalanalysis} and the different definition from the literature~\citep{bengio2006convex,bach2017breaking}. We will show that they are equal up to a factor $2$. This is a known result~\citep{neyshabur2014search}, and we provide here a natural proof using the measure theoretic formalism, for the sake of completeness. We stress that the analogous equivalence would fail for the RKHS norms, i.e., such a modification of the feature function could lead to different functional spaces.

Consider the feature functions $\phi(\theta,z)=c(a\cdot z)_+$ where $\theta = (a,c)\in \SS^{p-1}$ and $z=(x,1)$ (here we see $\SS^{p-1}$ as a subset of $\RR^{p-1}\times \RR$) and $\tilde \phi(a,z)=(a\cdot z)_+$ where $a\in \SS^{p-2}$. For a function $f:\SS^{p-2}\to \RR$, we consider two norms
\begin{align*}
\Vert f\Vert_{\Ff_1} &= \inf \left\{ \nu(\SS^{p-1}) \;;\; f(z) = \int \phi(\theta,z)\d\nu(\theta), \; \nu\in \Mm_+(\SS^{p-1})\right\}
\intertext{and}
\Vert f\Vert_{\tilde \Ff_1} &= \inf \left\{\vert \tilde \nu\vert(\SS^{p-2}) \;;\; f(z) = \int \tilde \phi(a,z)\d\tilde \nu(a), \; \nu\in \Mm(\SS^{p-2})\right\},
\end{align*} 
and the associated functional spaces $\Ff_1$ and $\tilde \Ff_1$ where these two norms are finite. Classical weak compactness arguments guarantee that the infimum defining these two norms is attained.

\begin{proposition}\label{prop:equivnorms}
It holds $\Ff_1=\tilde \Ff_1$ and for all $f\in \Ff_1$, it holds $\Vert f\Vert_{\Ff_1} =2 \Vert f\Vert_{\tilde \Ff_1}$. Moreover,  any measure that reaches the infimum for $\Vert \cdot\Vert_{\Ff_1}$ is concentrated on the set \[\left\{(a,c) \in \RR^{p-1}\times \RR\;;\; \Vert a\Vert =\vert c\vert = 1/\sqrt{2}\right\}.\]
\end{proposition}

An interesting consequence of this result is that empirical risk minimization with the commonly used \emph{weight decay} regularization and the total variation regularization used by~\citet{bengio2006convex, bach2017breaking} (the path-norm) are equivalent.

\begin{proof}
We give a constructive proof where we explicitly build a minimizer for each norm given a minimizer for the other norm. Let us start with a measure $\nu \in \Mm_+(\SS^{p-1})$ such that $f(x) = \int \phi(\theta,x)\d\mu(\theta)$. We define the linear operator $\Pi:\Mm_+(\SS^{p-1})\to \Mm(\SS^{p-2})$ where $\Pi(\nu)$ is characterized by
\[
\int \varphi \d\Pi(\nu) = \int c\Vert a\Vert\varphi(a/\Vert a\Vert)\d\nu((a,c)),
\]
where, as usual, the integrand is extended by continuity at $a=0$. By construction, it holds $\forall z\in \RR^{p-2}$,
\[
\int \tilde \phi (a,z)\d\Pi(\nu)(a) = \int c\Vert a\Vert \tilde \phi (a/\Vert a\Vert,z)\d\nu((a,c)) = \int \phi((a,c),z)\d\nu((a,c)) = f(z).
\]
As for the total variation norm $\Vert \Pi(\nu)\Vert := \vert \Pi(\nu)\vert(\SS^{p-2})$ of $\Pi(\nu)$, it can be bounded as follows. In the definition of $\Pi(\nu)$, we may restrict the integral over $\{c>0\}$, which defines a measure $\Pi_+(\nu) \in \Mm_+(\SS^{p-2})$. Similarly restricting the integral over $\{c<0\}$ an taking the opposite gives another measure $\Pi_-(\nu)\in \Mm_+(\SS^{p-2})$. It holds $\Pi(\nu) = \Pi_+(\nu)-\Pi_-(\nu)$ and thus $\Vert  \Pi (\nu)\Vert \leq \Vert\Pi_+(\nu)\Vert + \Vert \Pi_-(\nu)\Vert$.  Moreover, by integrating against $\varphi = 1$, it holds
\[
\Vert\Pi_+(\nu)\Vert = \int 1 \d\Pi_+(\nu) = \int_{c>0} c\Vert a\Vert \d\nu((a,c)) \leq \frac12 \int_{c>0} \d\nu((a,c)),
\]
since $c\Vert a\Vert \leq (\Vert a \Vert^2 + c^2)/2 =1/2$ for $(a,c)\in \SS^{p-1}$. Using a similar bound for $\Vert \Pi_-(\nu)\Vert$, we get that
\[
\Vert  \Pi (\nu)\Vert \leq \frac12 \int_{c>0} \d\nu((a,c))+\frac12 \int_{c<0} \d\nu((a,c)) \leq \frac12 \nu(\SS^{p-1}).
\]
Finally, tracking the equality cases, it holds $\Vert  \Pi (\nu)\Vert= \frac12 \nu(\SS^{p-1})$ if and only if $\nu$ is concentrated on the set given in Proposition~\ref{prop:equivnorms}, which is the intersection of the sphere with the set of points satisfying $2\vert c\vert \Vert a\Vert =\Vert a\Vert^2+ \vert c\vert^2$.

Conversely, let $\nu \in \Mm(\SS^{p-2})$ and consider its Jordan decomposition $ \nu = \nu_+ - \nu_-$ into two nonnegative measures, which is such that $\Vert \nu \Vert = \Vert  \nu _+\Vert + \Vert  \nu_-\Vert$.  We define two maps $T^+, T^-:\SS^{p-2}\to \SS^{p-1}$ as $T^+(a) = (a,1)/\sqrt{2}$ and $T^-(a) = (a,-1)/\sqrt{2}$. Now, define the linear map $T:\Mm(\SS^{p-2})\to \Mm_+(\Ss^{p-1})$ as
\[
T(\nu) = 2( T^+_\# \nu_+ + T^-_\# \nu_- ).
\]
Since pushforwards preserve the mass of nonnegative measures, it holds $\Vert T(\nu)\Vert = 2 (\Vert \nu_+\Vert +\Vert \nu_-\Vert) = 2\Vert \nu\Vert$. Moreover, using the definition of pushforward measures, we have
\begin{align*}
\int \phi(\theta,z) \d T(\nu)(\theta) &= 2 \int \phi((a,1)/\sqrt{2},z)\d\nu_+(a) + 2 \int \phi((a,-1)/\sqrt{2},z)\d\nu_-(a) \\
&=\int \tilde \phi(a,z)\d\nu_+(a) -\int \tilde \phi(a,z)\d\nu_-(a) = \int \tilde \phi(a,z)\d\nu(a) =f(z).
\end{align*}
To sum up, for any feasible measure $\nu$ for the definition of $\Vert f \Vert_{\Ff_1}$, we have built a measure $\Pi(\nu)$ that is feasible for $\Vert f \Vert_{\tilde \Ff_1}$ with a norm divided by at most $2$. Conversely, for any feasible measure $\nu$ for the definition of $\Vert f \Vert_{\tilde \Ff_1}$, we have built a measure $T(\nu)$ that is feasible for $\Vert f \Vert_{\Ff_1}$ with a norm multiplied exactly by $2$. This concludes the proof.
\end{proof}

\section{Details on Wasserstein gradient flows}\label{app:meanfield}
\subsection{Alternative formulations of the Wasserstein gradient flow}
\begin{itemize}
\item (Divergence form) It can be shown that a Wasserstein gradient flow as defined in Definition~\ref{def:WGF} satisfies, in the sense of distributions, the following partial differential equation~\citep{ambrosio2008gradient}
$$
\partial_t \mu_t = -\div(\nabla F'_{\mu_t} \mu_t).
$$
\item (Projected representation) If we look at the projected trajectory $\nu_t = \Pi_2(\mu_t)$, it can be shown that it solves the following dynamic, which is known as Wasserstein-Fisher-Rao or Hellinger-Kantorovich gradient flow of the functional $J:\Mm_+(\SS^{p-1})$ satisfying $J(\Pi_2(\mu))=F(\mu)$. In equation,
\begin{equation}\label{eq:projectedWGF}
\partial_t \nu_t = - \div(\nabla J'_{\nu_t}\nu_t) + 4J'_{\nu_t}\nu_t,
\end{equation}
where $J'_{\nu} (\theta) = \sum_{i=1}^n \nabla_i S(\hat h(\nu)) y_i \phi(\theta,x_i)$ is defined on the sphere, see~\citet{chizat2019sparse}.
Note that there is also a Lagrangian representation for this projected dynamics~\citep{maniglia2007probabilistic} .
\item (Renormalized dynamics) It can also be seen with a direct computation that the normalized dynamics $\bar \nu_t = \nu_t /\Vert \nu_t\Vert$ satisfies the following equation 
\begin{equation}\label{eq:normalizedWGF}
\partial_t \bar \nu_t = - \div(\nabla J'_{\nu_t}\bar \nu_t) + 4\left( J'_{\nu_t} - \int J'_{\nu_t} \d\bar \nu_t\right) \bar \nu _t.
\end{equation}
When the driving potential is $J'_{\bar \nu_t}$ (instead of $J'_{\nu_t}$), this dynamics is known as the spherical Wasserstein-Fisher-Rao or spherical Hellinger Kantorovich gradient flow~\citep{kondratyev2019spherical} and was considered by~\citet{rotskoff2019neuron} for neural networks training.
\end{itemize}

\subsection{Proof of Theorem~\ref{th:WGF}}
We just need to prove that the assumptions of~\citet[Theorem 2.6]{chizat2018global} are satisfied and justify that non-compactly supported initialization are also allowed.

\paragraph{Checking regularity assumptions.} With the notations used by~\citet[Assumptions 2.1]{chizat2018global}, the Hilbert space $\Ff$ is $\RR^n$, the domain ``$\Omega$'' is $\RR^{p}$, the risk ``$R$'' is the smooth-margin $S$, the function ``$\Phi$'' is here $\Phi(\theta) = (y_i \phi(\theta,x_i))_{i\in [n]}$, there is no regularization, and the family of nested sets ``$\Omega_r$'' are the closed balls  of radius $r$ in $\RR^p$. We can directly check that:
\begin{itemize}
\item under Assumption {\sf (A2)}, $S:\RR^n\to \RR$ is differentiable and its gradient given in coordinates by $\nabla_i S(u) = \ell'(-u_i)/(\sum_{i'} \ell(-u_{i'}))$ is Lipschitz continuous and bounded on superlevel sets;
\item under Assumption {\sf (A3)}, the function $\Phi$ is differentiable with a locally Lipschitz continuous gradient. Moreover its gradient has at most a linear growth by $2$-homogeneity (this verifies assumptions from~\citet[Assumptions 2.1-(iii)-(c)]{chizat2018global}).
\end{itemize}

\paragraph{Removing the compact support assumption.} \citet[Theorem 2.6]{chizat2018global} only allows an initialization in some ``$\Omega_{r}$'', which means here that $\mu_0$ should be compactly supported. However, when $\Phi$ is positively $2$-homogeneous, this condition can be relaxed~\citep[Appendix C.1]{chizat2019sparse} because the dynamics is entirely characterized by its projection on the sphere (through $\Pi_2$). Indeed, for any $\mu \in \Pp_2(\RR^p)$ distinct from $\delta_{0}$, there exists a compactly supported $\tilde \mu \in \Pp_2(\RR^p)$ such that $\Pi_2(\mu)=\Pi_2(\tilde \mu)$. Since we have existence and uniqueness for the Wasserstein gradient flow starting from $\tilde \mu$, we get existence and uniqueness from the Wasserstein gradient flow $(\mu_t)_t$ starting fom $\mu$ (since it is entirely determined by the velocity field $\nabla F'_{\mu_t}$ which is itself determined by the projection of the dynamics $\Pi_2(\mu_t)$. Remark that with similar arguments, we can adapt the proof of~\citet[Theorem 2.6]{chizat2018global} to any initialization $\mu \in \Pp_2(\RR^p)$.

\section{Appendix to Section~\ref{sec:implicitbias}: main theorem}\label{app:maintheorem}

\subsection{Proof of the main theorem}
Let us restate Theorem~\ref{th:biasWGF} and give its proof. We recall that $\nu_t := \Pi_2(\mu_t)$ and $\bar \nu_t := \nu_t/(\nu_t(\SS^{p-1}))$.
\begin{theorem}
Under {\sf (A1-3)}, assume that $\Pi_2(\mu_0)$ has full support on $\SS^{p-1}$, that $\nabla S(h(\mu_t))$ converges in $\RR^n$ and that $\bar \nu_t$ converges weakly towards some $\bar \nu_{\infty}$. Then $\bar \nu_{\infty}$ is a maximizer of Eq.~\eqref{eq:maxmargin1}.
\end{theorem}

\begin{proof}
By Lemma~\ref{lem:limitgradient} the limit $p(\infty)\in \RR^n$ of $p(t) \coloneqq \nabla S(\hat h(\mu_t))$ is non-zero. This implies that $J'_{\nu_t}$, the restriction of $F'_{\nu_t}$ to the sphere, converges in $\Cc^1(\SS^{p-1})$ (i.e.~the function and its gradient converge uniformly) to a function $J'_\infty: \theta \mapsto \sum_{i=1}^n p_i(\infty)y_i\phi(\theta,x_i)$ which is non-zero because the family $(\phi(\cdot,x_1),\dots,\phi(\cdot,x_n))$ is linearly independent by {\sf (A3)}. Since $\phi$ is balanced by {\sf (A1)}, we have $M\coloneqq \max_{\theta \in \SS^{p-1}} J'_\infty(\theta) >0$. The rest of the proof is divided into $3$ steps. 

\textbf{Step 1: mass grows unbounded.} In a first step, we prove that $\nu_t(\SS^{p-1})\to \infty$. Assume that $J'_\infty$ is not constant (the other case will be considered later), and let $v \in {]0,M/8[}$ be such that $M-v$ is a regular value of $J'_\infty$, i.e., be such that $\Vert \nabla J'_{\infty}\Vert$ does not vanish on the $M-v$ level-set of $J'_\infty$. Such a $v$ is guaranteed to exist thanks to the fact that  $\phi(\cdot,x_i)$ is subanalytic (which implies that $J'_{\infty}$, which is a finite sum of such $\phi(\cdot,x_i)$, is also subanalytic) and that the sphere is a subanalytic set, and then applying~\citep[Thm.~14]{bolte2006nonsmooth}. Note that such admissible $v$ are dense in the range of $J'_\infty$, which will be useful in Step.~3. Let $K_v = (J'_{\infty})^{-1}([M-v,M])\subset \SS^{p-1}$ be the corresponding super-level set. By the regular value theorem, the boundary $\partial K_v$ of $K_v$ is a differentiable orientable compact submanifold of $\SS^{p-1}$ and is orthogonal to $\nabla J'_{\infty}$. 
By construction, it holds for all $\theta \in K_v$, $J'_{\infty}(\theta) \geq M-v$ and, for some $u>0$, by the regular value property, $\nabla J'_\infty(\theta)\cdot \vec n_{\theta} \geq u$ for all  $\theta \in \partial K_v$ where $\vec n_\theta$ is the unit normal vector to $\partial K_v$ at $\theta$ pointing inwards. Since $J'_{\nu_t}$ converges in $\Cc^1(\SS^{p-1})$ towards $J'_\infty$, there exists $t_0>0$ such that for all $t\geq t_0$, $\Vert J'_{\nu_t}-J'_\infty\Vert_{\Cc^1(\SS^{p-1})} \leq \min\{v,u/2\}$ and thus
\begin{align*}
\forall \theta \in K_v, \quad  J'_{\nu_t}(\theta) \geq m-2v &&\text{and}&& \forall \theta \in \partial K_v,\quad \nabla J'_{\nu_t}(\theta)\cdot \vec n_{\theta} \geq u/2.
\end{align*}
This second property guarantees that no mass leaves $K_v$ due to the divergence term in Eq.~\eqref{eq:projectedWGF} for $t\geq t_0$ (equivalently, due to the fact that $\RR_+ K_v$ is a positively invariant set of the flow $X$ for $t\geq t_0$). Thus, taking into account the reaction/growth term in Eq.~\eqref{eq:projectedWGF}, it holds for $t\geq t_0$,
\[
\frac{\d}{\d t} \nu_t(K_v) \geq 4 \int_{K_v} J'_{\nu_t}\d\nu_t  \geq 4(M-2v)  \nu_t(K_v).
\]
It follows by Gr\"onwall's lemma that $\nu_t(K_v)\geq \exp(4(M-2v)t)\nu_{t_0}(K_v)$ for $t\geq t_0$. On the other hand, $\nu_{t_0}$ has full support on $\SS^{p-1}$ since it can be written as the pushforward of a rescaled version of $\nu_0$ by a diffeomorphism, see~\citet[Eq. (1.3)]{maniglia2007probabilistic} (this is the only place where the assumption on the support is needed). Thus $\nu_{t_0}(K_v)>0$ and it follows that $\nu_t(\SS^{p-1})\to \infty$. To deal with the case where $J'_\infty$ is constant and is equal to $M>0$, we can directly take $K=\SS^{p-1}$ to show that $\nu_t(\SS^{p-1})\to \infty$. In the rest of the proof, we show that $(\bar \nu_\infty,p(\infty))$ satisfy the optimality conditions of Eq.~\eqref{eq:maxmargin1} given by Proposition~\ref{prop:optimality}, which we refer to as the \emph{complementary slackness} conditions.

\textbf{Step 2: complementary slackness (I).} We first show that $\min_i \hat h_i \to \infty$. Using the property of gradient flows and previously established estimates, we have for $T>0$,
\begin{align*}
F(\mu_T) - F(\mu_{0}) = \! \int_0^T \!\Vert \nabla F'_{\mu_t}\Vert^2 \d\mu_t \geq \! \int_{t_0}^T \! \int_{K_v} \! 4\vert J'_{\nu_t}\vert^2 \d\nu_t \d t \geq 4(M-2v)^2\int_{t_0}^T\! \nu_t(K_v)\d t \to \infty.
\end{align*}
Thus $F(\mu_t)\to \infty$ which implies that for all $i\in [n]$, $\ell(-\hat h_i(\mu_t)) \to 0$ and thus $\hat h_i(\mu_t)\to \infty$. Applying Lemma~\ref{lem:argmax}, it follows that $p_{i_0}(\infty) =0$ for all $i_0\in \arg\min_i \hat h_i(\nu_\infty)$.

\textbf{Step 3: complementary slackness (II).} We now show that  $\bar \nu_\infty$ is concentrated on $(J'_\infty)^{-1}(M)$, where $\bar \nu_t = \nu_t/\Vert \nu_t\Vert \in \Pp(\SS^{p-1})$ is the normalized path and $\bar \nu_\infty$ its limit. This is immediate if $J'_\infty$ is constant. Otherwise, assuming that $M-4v$ is also a regular value of $J'_\infty$ and taking a potentially smaller $u$ (which can always be achieved by perturbing $v$ if needed, since regular values are dense as mentioned in Step.~1), it holds for $t\geq t_0$,
\[
\frac{\d}{\d t} \nu_t(\SS^{p-1} \setminus K_{4v}) \leq 4 \int_{\SS^{p-1} \setminus K_{4v}} J'_{\nu_t}\d\nu_t  \leq 4(m-3v)\nu_t(K_{4v})
\]
using the fact that no mass enters into $\SS^{p-1}\setminus K_{4v}$ due to the divergence term in Eq.~\eqref{eq:projectedWGF} for $t\geq t_0$.
Comparing the rate of growth of the mass in $K_v$ and in $\SS^{p-1}\setminus K_{4v}$, we get that $\bar \nu_\infty(\SS^{p-1}\setminus K_{4v}) \leq \lim_{t\to\infty} \bar \nu_t(\SS^{p-1}\setminus K_{4v}) = 0$ since $\SS^{p-1}\setminus K_{4v}$ is open and by the properties of weak convergence of measures (Portmanteau Theorem). Since this holds for $v$ arbitrarily close to $0$, it follows that $\bar \nu_\infty$ is concentrated on $(J'_\infty)^{-1}(M)$.

\textbf{Step 4: conclusion.} We have proved the two complementary slackness properties, so by Proposition~\ref{prop:optimality}, the pair $(\bar \nu_\infty,p(\infty))$ satisfies the optimality conditions, which concludes the proof.
\end{proof}

\subsection{Proof of Corollary~\ref{cor:interchange}}
\begin{corollary}
Under the assumptions of Theorem~\ref{th:biasWGF}, assume that the sequence $(w_j(0))_{j \in \NN_*}$ is such that $\mu_{0,m}$ converges in $\Pp_2(\RR^p)$ to $\mu_0$. Then, denoting $\bar \nu_{m,t} = \Pi_2(\mu_{m,t})/[\Pi_2(\mu_{m,t})](\SS^{p-1})$, it holds
\[
\lim_{m,t\to \infty} \left(\min_{i\in [n]} y_i \int \phi(\theta,x_i)\d\bar \nu_{m,t}\right) = \gamma_1.
\]
\end{corollary}

\begin{proof}
The fact that 
$$
\lim_{t\to \infty}  \lim_{m\to \infty} \left(\min_{i\in[n]}y_i \int \phi(\theta,x_i)\d\bar \nu_{m,t}(\theta)\right) = \gamma_1
$$
is obtained by combining Theorem~\ref{th:WGF} with Theorem~\ref{th:biasWGF} because $\int \phi(\theta,\cdot)\bar \nu_{m,t}$ depends continuously on $\mu_{t,m}$ in $\Pp_2(\RR^p)$ (endowed with the Wasserstein distance $W_2$), so we only need to prove that limits can be interchanged. We detail the proof for $\ell=\exp$, noticing that we only use the asymptotic behavior of $\ell$ and $\ell'$ so it extends to any loss satisfying~{\sf (A2)}. Intuitively, in order to prove the other limit, we need to show that if for some $(t_0,m_0)$ the classifier is close to the max-margin classifier, then this remains true for $(t,m_0)$, $t\geq t_0$. First, for $\beta=\Vert \nu_t\Vert>0$, it holds at time $t$ (considering $\beta$ fixed):
\[
\frac{d}{dt} S_{\beta}(\hat h(\bar \nu_t)) = \int \Vert \nabla J'_{\nu_t}\Vert^2 \d\bar \nu_t + 4\left( \int\vert J'_{\nu_t}\vert ^2\d\bar \nu_t  - \left(\int\vert J'_{\nu_t}\vert \d\bar \nu_t\right)^2\right) \geq 0.
\]
On the other hand, direct computations using the fact that $\nabla S(u) \in \Delta^{n-1}$ leads to
$$
\partial_\beta S_\beta(u) = - \frac{1}{\beta^2} \sum_{i=1}^n \nabla_i S(u) \log(n\nabla_i S(u)) \geq - \frac{n}{\beta^2}.
$$
Combining both gives the total derivative
\begin{align*}
\frac{\d}{\d t} S_{\Vert \nu_t\Vert}(\hat h(\bar \nu_t)) \geq -n\frac{\d}{\d t}\left( \frac{1}{\Vert \nu_t\Vert}\right) && \Rightarrow && S_{\Vert \nu_t\Vert}(\hat h(\bar \nu_t)) \geq S_{\Vert \nu_{t_0}\Vert}(\hat h(\bar \nu_{t_0})) - \frac{n}{\Vert \nu_{t_0}\Vert}.
\end{align*}
From the first limit above, for any $\epsilon>0$, there exists $(t_0,m_0)$ such that for all $m\geq m_0$, $\min_{i\in [n]} y_i \int \phi(\theta,x_i)\d\bar \nu_{m,t_0}(\theta)  \geq \gamma_1 -\epsilon/3$. We have shown in the proof of Theorem~\ref{th:biasWGF} that $\Vert \nu_{t}\Vert \to \infty$ so choosing $t_0$ and $m_0$ potentially larger if needed, since $\Vert \nu_{t,m}\Vert \to \Vert \nu_t\Vert$ it also holds $\Vert \nu_{t_0,m}\Vert \geq 3n/\epsilon$ for  $m\geq m_0$. By the inequality above, we thus have for all $t\geq t_0$ and $m\geq m_0$
\[
S_{\Vert  \nu_{t,m}\Vert}(\hat h(\bar \nu_{t,m}))\geq S_{\Vert  \nu_{t_0,m}\Vert}(\hat h(\bar \nu_{t_0,m})) - \epsilon/3.
\]
Moreover, by Lemma~\ref{lem:softmax}, we have $\vert S_{\Vert \nu_t\Vert}(u) - \min_i u \vert \leq  \epsilon/3$ for $\Vert \nu_t\Vert$ large enough, uniformly for $u$ in a compact set. Hence  for all $m\geq m_0$,
\[
 \lim_{t\to \infty} S_{\Vert  \nu_{t,m}\Vert}(\hat h(\bar \nu_{t,m}))
\geq  S_{\Vert  \nu_{t_0,m}\Vert}(\hat h(\bar \nu_{t_0,m})) - \epsilon/3 \geq \gamma_1 -\epsilon.
\]
It remains to show that for all $C>0$, there exists $m_0$ such that if $m\geq m_0$ then $\lim \inf_{t\to\infty} \Vert \nu_{t,m}\Vert\geq C$, so that we deduce from the above
$$
\lim_{t\to \infty} \min_{i\in [n]} \hat h(\bar \nu_{t,m})[i] \geq \gamma_1-2\epsilon.
$$
Since $\epsilon$ is arbitrary, it would follow that $\lim_{m\to \infty}\lim_{t\to \infty} \min_{i\in [n]} \hat h(\bar \nu_{\infty,t}) \geq \gamma_1$ which is our claim.

To see this, it is sufficient to notice that since $F(\nu_{t,m})$ is increasing in $t$ for all $m$, and $S(u)-\log(n) \leq \min_i u$, we have that for all $C>0$, there exists $t_0,m_0$ such that $\min_i \hat h(\nu_{t,m})[i]\geq C$ for all $t\geq t_0$ and $m\geq m_0$.
\end{proof}

\subsection{Intermediate results}
This section contains intermediate results used in the proof of Theorem~\ref{th:biasWGF}.
\begin{proposition}[Optimality conditions]\label{prop:optimality}
The maximization problem~\eqref{eq:maxmargin1} admits global maximizers $\nu^\star \in \Mm_+(\SS^{p-1})$. Moreover, a measure $\nu^\star \in \Mm_+(\SS^{p-1})$ is a global maximizer of~\eqref{eq:maxmargin1} if and only if $\nu^\star \in \Pp(\SS^{p-1})$ and there exists $p^\star \in \Delta^{n-1})$ such that (i) $\spt \nu^\star \subset \arg\max_\theta \sum_i p_iy_i \phi(\theta,x_i)$ and (ii) $\spt p \subset \arg\min_i y_i \int \phi(\theta,x_i)\d\nu^\star(\theta)$.
\end{proposition}
\begin{proof}
By minimax duality~\citep{sion1958general}, we can rewrite Eq~\eqref{eq:maxmargin1} as the minimax problem
\[
\sup_{\nu \in \Pp(\Theta)}\inf_{p\in \Delta^{n-1}} \sum_{i=1}^n p_i y_i \int \phi(\theta,x_i)\d\nu(\theta) = \inf_{p\in \Delta^{n-1}} \sup_{\nu \in \Pp(\Theta)} \sum_{i=1}^n p_i y_i \int \phi(\theta,x_i)\d\nu(\theta) ,
\]
and it admits (at least) a saddle point $(\nu^\star,p^\star)$. Moreover, the optimality conditions are necessary and sufficient for the right-hand side to equal the left-hand side.
\end{proof}
Let us now prove some useful properties of the function
\begin{equation}
S_{\beta}(u)= - \frac1\beta \log\left( \frac1n \sum_{i=1}^n \ell(- \beta u_i)\right),
\end{equation}
which is a soft-min function under assumption {\sf (A2)}, as shown in the next lemma.
\begin{lemma}\label{lem:softmax}
If $\ell(u) \sim \exp(u)$ as $u\to \infty$, and if $\bar u \in (\RR_+^*)^n$, then
\[
\lim_{\beta \to \infty} S_\beta(\bar u) = \min_{i\in [n]} \bar u.
\]
\end{lemma}
\begin{proof}
Let $m \coloneqq \min_{i\in [n]} \bar u_i$. We have
$\exp\left( -\beta \left( S_{\beta}(\bar u) - m \right) \right) = \frac1n \sum_{i=1}^n \ell(-\beta \bar u_i )\exp(\beta m )$, where each term satisfies, in the large $\beta$ regime,
\[
\ell(-\beta \bar u_i )\exp(\beta m) \sim \exp(-\beta (\bar u_i -m)) \to \begin{cases}
1 & \text{if $i\in \arg\min_i \bar u_i$,}\\ 0 &\text{otherwise.}
\end{cases}
\]
As a consequence, $\exp\left( -\beta \left( S_{\beta}(\bar u) - m \right) \right) \to\frac1n  \# \arg\min_i \bar u_i \in {]0,1]}$. Thus, $S_{\beta}(\bar u) \to m$.
\end{proof}

The next lemma is immediate when $\ell=\exp$, but worth detailing for the general case.
\begin{lemma}\label{lem:limitgradient}
Under Assumption  {\sf (A2)}, let $u(t)$ be a sequence such that $S(u(t))$ is lower bounded and $\nabla S(u(t))$ converges. Then  $\lim_{t\to \infty} \nabla S(u(t)) \neq 0$.
\end{lemma}
\begin{proof}
We analyze separately two cases, whether there is $i_0\in [n]$ such that $u_{i_0}(t)$ is upper bounded or not. In the first case, we have
\[
\nabla_{i_0} S(u(t)) = \frac{\ell'(-u_{i_0}(t))}{\sum_{i}\ell(-u_i(t))} = \frac1n \ell'(-u_{i_0}(t)) \exp(S(u(t))),
\]
which is uniformly lower bounded by a positive constant under  {\sf (A2)}  and due to the lower bound on $S(u(t))$ hence $\lim_{t\to \infty} \nabla S_{i_0}(u(t))\neq 0$. In the other case, up to taking a subsequence (which does not change the limit), we can assume that $u_i(t)\to \infty$ for all $i\in [n]$. Then using the equivalent of $\ell$ and $\ell'$ at $-\infty$, we have that $\lim_{t\to \infty} \sum_{i\in [n]}  \nabla S(u(t))=1$ which is sufficient to conclude. 
\end{proof}

The next lemma is adapted from~\citet[Lemma 8]{gunasekar2018characterizing} and exploits the fact that the gradient of $S$ is a soft-argmax.
\begin{lemma}[Convergence of soft-argmin]\label{lem:argmax}
Let $\beta:\RR_+\to \RR_+$ and $\bar u:\RR_+\to \RR^n$ be such that $\beta(t) \to \infty$ and $\bar u(t)  \to \bar u(\infty)\in (\RR_+^*)^n$ as $t\to \infty$. If $\ell(-u)\sim \ell'(-u)\sim \exp(-u)$ as $u\to \infty$, then for any $i_0\notin \arg\min_i \bar u_i(\infty)$, as $t\to \infty$, it holds $\nabla_{i_0} S_{\beta(t)}(\bar u(t)) \to 0$. 
\end{lemma}
\begin{proof}
Taking any $\gamma \in {]\min_{i} \bar u_i(\infty), \bar u_{i_0}(\infty)[}$, we have
\[
\nabla_{i_0} S_{\beta(t)}(\bar u(t)) = \frac{\ell'(-\beta \bar u_{i_0}(t))}{\sum_{i=1}^n \ell(-\beta \bar u_i(t))} =  \frac{\ell'(-\beta \bar u_{i_0}(t))\exp(\beta \gamma)}{\sum_{i=1}^n \ell(-\beta \bar u_i(t))\exp(\beta \gamma)}.
\]
By the assumption on $\ell'$, the numerator is equivalent to $\exp(-\beta (\bar u_{i_0}(t)-\gamma))$ and goes to $0$ as $t\to \infty$. Also, by the assumption on $\ell$, each of the term in the denominator is equivalent to $\exp(-\beta (\bar u_{i}(t)-\gamma))$ which goes to $\infty$ for $i \in \arg\min_i \bar u_i(\infty)$, hence the conclusion.
\end{proof}

\section{Appendix for Section~\ref{sec:rate}}\label{app:ratesL1}

Let us define 
$$
G_\beta(a) =  -\frac1\beta \log\left( \frac1n\sum_{i=1}^n \exp\left(-  \beta \sum_{j=1}^m z_{i,j} a_j\right)\right),
$$
which satisfies
\begin{equation}\label{eq:comparison}
\min_{i\in [n]} z_i^\top a \leq G_{\beta}(a) \leq \min_{i\in [n]} z_i^\top a + \frac{\log n}{\beta}.
\end{equation}

Let us recall Proposition~\ref{prop:OMA} and prove it.
\begin{proposition}
Let $a_j(t) = r_j(t)^2/m$ for $j\in [m]$, $\beta(t) = \Vert a(t)\Vert_1$ and $\bar a(t) = a(t)/\beta(t)$. For the step-sizes $\eta(t) = 1/(16\Vert z \Vert_\infty \sqrt{t+1})$ and a uniform initialization $r(0)\propto \mathbf{1}$, it holds
\begin{equation}\label{eq:intermediatebound}
\max_{0\leq s\leq t-1} \min_{i} z_i^\top \bar a(s)  \geq \gamma_1^{(m)} - \frac{\Vert z\Vert_\infty}{\sqrt{t}}(8\log(m) + \log(t)+1) - \frac{4\log n}{\sqrt{t}}  \sum_{s=0}^{t-1} \frac{1}{\beta(s)\sqrt{s+1}}.
\end{equation}
where $\gamma_1^{(m)} := \max_{a \in \Delta^{m-1}}\min_{i \in [n]} z_i^\top a$ and the last sum is uniformly bounded in $t$ as soon as $\gamma_1^{(m)}>0$.
\end{proposition}

\begin{proof}
Let us first prove that the normalized dynamics $\bar a(t)$ satisfies the (perturbed) \emph{online mirror ascent} recursion (on the simplex, with the entropy mirror map):
$$
\left\{
\begin{aligned}
 b_j(t+1) &= \bar a_j(t)\left( 1 + 4\eta(t) \nabla_j G_{\beta(t)}(\bar a(t)) +4\eta(t)^2 \vert \nabla_j G_{\beta(t)}(\bar a(t))\vert^2 \right), \; \forall j\in [m]\\
\bar a(t+1) &= b(t+1)/ \Vert b(t+1)\Vert_1.
\end{aligned}
\right.
$$
We mention that it is perturbed because for the plain \emph{online mirror ascent}, the multiplicative term in the first line would be $\exp(4\eta(t)\nabla G_{\beta(t)}(\bar a (t))$, so we have second order corrections in $\eta(t)$. 

Since $\nabla_j F(r(t)) = (2/m) r_j(t) \nabla G_1(a(t))$, it follows
\begin{align*}
a_j(t+1) &= (r_j(t) +2\eta(t) r_j(t) \nabla _j G_1(a(t)))^2\\
&= r_j(t)^2 + 4\eta(t) r_j(t)^2 \nabla _j G_1(a(t)) + 4\eta(t)^2 r_j(t)^2 \vert \nabla _j G_1(a(t))\vert^2\\
& = a_j(t) \left(1 +  4\eta(t) \nabla _j G_1(a(t)) + 4\eta(t)^2 \vert \nabla _j G_1(a(t)\vert^2\right).
\end{align*}
Now using the fact that $\nabla G_1(a) = \nabla G_{\Vert a\Vert_1}(a/\Vert a\Vert_1)$, it follows
\begin{align*}
\bar a_j(t+1) = \frac{\Vert a(t)\Vert_1}{\Vert a(t+1)\Vert_1} \bar a_j(t) \left(1 +  4\eta(t) \nabla_j G_{\beta(t)}(\bar a(t)) + 4\eta(t)^2 \vert \nabla_j G_{\beta(t)}(\bar a(t)\vert^2\right),
\end{align*}
hence the iterations. Let us rewrite these iterations using the framework of Bregman divergences (see~\citet[Chap. 4]{bubeck2015convex} for details). For $a,b \in \RR_+^m$, let $\phi(a) =\sum_{i=1}^{m} a_j\log(a_j)-a_j$, let $D(a,b) = \phi(a)-\phi(b) -\nabla \phi(b)^\top(a-b)$ and let $\Pi(a)=a/\Vert a\Vert_1 = \arg\min_{b\in \Delta^{m-1}} D(b,a)$ be the Bregman projection on the simplex for the divergence $D$. With $g(s) = \nabla G_{\beta(s)}(\bar a(s))$, we have 
$$
\left\{
\begin{aligned}
 \nabla \phi(b(t+1)) &= \nabla\phi(\bar a(t)) + \eta(t) g(t) +\eta_t^2 e(t) \\
\bar a(t+1) &= \Pi(b(t+1)),
\end{aligned}
\right.
$$
which are the online mirror ascent updates, with a second-order error term $e(t)_j = \log(1+4\eta(s)g_j(t)+4\eta(t)^2g_j(t)^2)-4\eta(t)g_j(t)$.  
Notice that $\forall t$, $\Vert g(t)\Vert_\infty \leq \Vert z\Vert_\infty$ so if we assume that $\eta(t)\leq 1/(16\Vert z\Vert_\infty)$ we have $\eta(t)\Vert g(t)\Vert_\infty\leq 1/16$. Using the inequality $\vert \log(1+u)- u\vert \leq u^2$ for $\vert u\vert \leq 1/2$, we get by applying it with $u_j = 4\eta(t)g_j(t)+4\eta(t)^2g_j(t)^2$,
$$
\Vert e(s)\Vert_\infty \leq (4\eta(t)\Vert z\Vert_\infty+4\eta(t)^2\Vert z\Vert_\infty^2)^2 + 4\eta(t)^2\Vert z\Vert_\infty^2 \leq 23 \eta(t)^2\Vert z\Vert_\infty^2.
$$

We now follow the usual proof of mirror ascent from~\citet[Thm. 4.2]{bubeck2015convex} (or~\citet{beck2003mirror} for the variable step-size case) and including this error term leads to, for all $\bar a^*\in \Delta^{m-1}$,
$$
4\eta(t)g(t)^\top (\bar a^*-\bar a(t))  \leq D(\bar a^*,\bar a(t)) - D(\bar a^*,\bar a(t+1)) + 24\eta(t)^2\Vert z\Vert_\infty^2.
$$
We get a telescopic sum and using the concavity of each $G_{\beta}$,
$$
S(t) := \sum_{s=0}^{t-1} \eta(s)(G_{\beta(s)}(\bar a^*) - G_{\beta(s)}(\bar a(s))) \leq  \frac14 D(\bar a^*,\bar a(0)) + 6 \Vert z\Vert_\infty^2 \sum_{s=0}^{t-1} \eta(s)^2.
$$
With our choice of initialization, $D(\bar a^*,\bar a(0)) \leq \log(m)$. Let us choose $\eta(t) = \tau/\sqrt{t+1}$. Using the inequalities
$$
\sum_{s=0}^{t-1}\frac{1}{\sqrt{s+1}} \geq \int_1^{t+1}\frac{ds}{\sqrt{s}} = 2\sqrt{t+1}-2.
$$
and 
$$
\sum_{s=0}^{t-1} \left(\frac{1}{\sqrt{s+1}}\right)^2 = \sum_{s=1}^t \frac{1}{s} \leq 1 +\int_{s=1}^{t}\frac{ds}{s} =1+\log(t).
$$
It follows that for all $t\geq 1$,
$$
S(t) :=\frac{ \sum_{s=0}^{t-1} \eta(s)(G_{\beta(s)}(\bar a^*) - G_{\beta(s)}(\bar a(s)))}{\sum_{s=0}^{t-1} \eta(s)}  \leq \frac{\log(n)/4 + 6\tau^2 \Vert z \Vert_\infty^2 (1+\log(t))}{2\tau(\sqrt{t+1}-1)}.
$$
In particular, with the choice $\tau = 1/(16\Vert z\Vert_\infty)$, we get 
$$
S(t) \leq \frac{\Vert z\Vert_\infty}{\sqrt{t+1}-1}\Big( 2\log(m) + (1+\log(t))/4 \Big) \leq \frac{\Vert z\Vert_\infty}{\sqrt{t}}(8\log(m) + \log(t)+1).
$$
where we used $\sqrt{t}/4\leq \sqrt{t+1}-1$ to simplify the expression.
Finally, using inequality~\eqref{eq:comparison}, we have
$$
G_{\beta(s)}(\bar a^*) - G_{\beta(s)}(\bar a(s)) \geq \min_{i} z_i^\top \bar a^* - \min_{i} z_i^\top \bar a(s) - \frac{\log n }{\beta(s)}.
$$
Taking the weighted sum gives
\begin{align*}
\gamma_1^{(m)}- \max_{0\leq s\leq t-1} \min_{i} z_i^\top \bar a(s)  &\leq S(t) + \frac{ \sum_{s=0}^{t-1} \eta(s) (\log n) /\beta(s)}{\sum_{s=0}^{t-1} \eta(s)} \\
& \leq \frac{\Vert z\Vert_\infty}{\sqrt{t}}(8\log(m) + \log(t)+1) + \frac{4\log n}{\sqrt{t}}  \sum_{s=0}^{t-1} \frac{1}{\beta(s)\sqrt{s+1}}.
\end{align*}
The conclusion follows by Lemma~\ref{lem:betaunbounded}.
\end{proof}

In the next result, we show that the norm of the iterates grows to $+\infty$ and that  $\sum_{s=0}^{t-1} \frac{1}{\beta(s)\sqrt{s+1}}$ is finite. For simplicity, we do not track the constants.
\begin{lemma}\label{lem:betaunbounded} Under the assumptions of Proposition~\ref{prop:OMA}, we have that $\beta(t)\to \infty$ and $\sum_{s=0}^{t-1} \frac{1}{\beta(s)\sqrt{s+1}}$ is bounded uniformly in $t$.
\end{lemma}
\begin{proof}
For this result, we look at a different online mirror ascent dynamics. We consider  $\alpha>0$ (to be chosen appropriately later) and define $\tilde \beta(t) = \max\{ 1, \max_{0\leq s\leq t}\{\beta(s)/\alpha\}\}$ and the iterates $\tilde a(t) = a(t)/\tilde \beta(t)$. With the same arguments than in the proof of  Proposition~\ref{prop:PGD}, it can be seen that those iterates satisfy the recursion, with $g(t)=\nabla G_{\tilde \beta(t)}(\tilde a(t))$,
$$
\left\{
\begin{aligned}
b_j(t+1) &= \tilde a_j(t) (1+4\eta(t) g_j(t) + 4\eta(t)^2 g_j(t)^2).\\
\tilde a(t+1) &= b(t+1)/\Vert b(t+1)\Vert .
\end{aligned}
\right.
$$
These are (perturbed) online mirror ascent iterates for the sequence of losses $G_{\tilde \beta(t)}$, step-sizes $\eta_t$ on the set $\alpha B^1_+ =\{ \tilde a \in \RR^{m}_+\;;\; \sum_{j} \tilde a_j \leq \alpha \}$. Note that the entropy $\phi(s)=s\log s -s +1$ is $1/\alpha$ strongly convex with respect to $\Vert \cdot\Vert_1$ on this set, and that $\Vert g\Vert_{\infty}$ is bounded uniformly  in $\alpha$. We have the usual mirror descent bound~\citep[Chap. 4][]{bubeck2015convex} with $\bar a^*$ the $\ell_1$-max-margin solution,
$$
\eta(t)g(s)^\top (\alpha \bar a^* - \tilde a(t)) \leq H(\alpha \bar a^*, \tilde a(t+1)) - H(\alpha \bar a^*, \tilde a(t)) + \alpha C \eta(t)^2
$$
where $C$ only depend on $\Vert z\Vert_\infty$. By summing we get 
\begin{align*}
S_\alpha(t) := \frac{\sum_{s=0}^{t-1} \eta(s)g(s)^\top (\alpha \bar a^* - \tilde a(s))}{\sum_{s=0}^{t-1} \eta(s)} 
& \leq \frac{H(\alpha \bar a^*,\tilde a(0)) + C\alpha \sum_{s=0}^{t-1} \eta(s)^2}{\sum_{s=0}^{t-1} \eta(s)}\\
& \lesssim \frac{\alpha \log(\alpha) +\log(t)}{\sqrt{t+1}-1},
\end{align*}
where we only track the dependency in $t$ and $\alpha$. On the other hand, using inequality~\eqref{eq:comparison},
\begin{align*}
S_\alpha(t) & \geq  \frac{\sum_{s=0}^{t-1} \eta(s)(G_{\tilde \beta(s)}(\alpha \bar a^*)-G_{\tilde \beta(s)}(\tilde a(s)))}{\sum_{s=0}^{t-1} \eta(s)}  \\
&\geq \alpha \gamma_1^{(m)} - \log(n) \Big(\frac{\sum_{s=0}^{t-1} \eta(s)/\tilde \beta(s)}{\sum_{s=0}^{t-1} \eta(s)}\Big) - \gamma_1^{(m)} \Big(\frac{\sum_{s=0}^{t-1} \eta(s) \Vert  \tilde a(s)\Vert_1}{\sum_{s=0}^{t-1} \eta(s)}\Big).
\end{align*}
Thus, using the fact that $\Vert \tilde a(s)\Vert_1\leq \Vert a(s)\Vert = \beta(s)$, and that $\tilde \beta(s)\geq 1$, it follows
$$
\frac{\sum_{s=0}^{t-1} \eta(s) \beta(s)}{\sum_{s=0}^{t-1} \eta(s)} \geq \alpha - \frac{\log n}{\gamma_1^{(m)}} - S_\alpha(t).
$$
As a consequence
$$
\sum_{s=0}^{t-1} \frac{\beta(s)}{\sqrt{s+1}} \gtrsim \sqrt{t+1} \left( \alpha  -  \frac{\alpha \log(\alpha) +\log(t)}{\sqrt{t+1}-1}\right).
$$ 
Taking for instance $\alpha = \sqrt{t}$ shows that $\sum_{s=0}^{t-1} \frac{\beta(s)}{\sqrt{s+1}} \gtrsim t - \log(t)$ and thus $\beta(t)\to \infty$.
By Lemma~\ref{lem:asympbeta1} then $\beta(t)$ is increasing for $t\geq t_0$ and grows to $\infty$ at a super-polynomial rate and the conclusion follows.
\end{proof}

We now prove the asymptotic rate of growth of the norm of the iterates.
\begin{lemma}\label{lem:asympbeta1}
If $\eta(t) \asymp 1/\sqrt{t+1}$ and $\beta(t)\to \infty$, then $\beta(t)$ is increasing for $t$ large enough and
$$
\log(\beta(t))\gtrsim  \min\{1, \gamma_1^{(m)}\} \sqrt{t}.
$$
\end{lemma}
\begin{proof}
We have for all $t$,
$$
\frac{\beta(t+1)}{\beta(t)} = 1 +\eta(t) \nabla S(Za(t))^\top Z\bar a(t) + O(1/(t+1)).
$$
By Eq.~\eqref{eq:intermediatebound} (which holds irrespective of this lemma), we know that $\beta(t)\to \infty$ implies $\bar a(t) \to \bar a^*$ where $\bar a^*$ is the $\ell_1$-max-margin solution. Since $\nabla S(Za(t)) \in \Delta^{n-1}$, it holds for $t$ large enough
$$
\frac{\beta(t+1)}{\beta(t)} \geq 1 +\frac12 \gamma_1^{(m)} \eta(t)+ O(1/(t+1)).
$$
Taking the logarithm and summing, we get
$$
\log(\beta(t)) -\log(\beta(0))\geq \sum_{s=0}^{t-1}\left(\frac12 \gamma_1^{(m)}\eta(s) +O(1/s) \right) = \frac12 \gamma_1^{(m)} \sum_{s=0}^{t-1} \eta(s) + O(\log(t)).
$$
The result follows since $\sum_{s=0}^{t-1}(t+1)^{-1/2} \geq 2(\sqrt{t-1}-1)$.
\end{proof}

\section{Appendix to Section~\ref{sec:convex}}\label{app:output}

Let us recall Proposition~\ref{prop:PGD} and prove it.

\begin{proposition}
Let $a(t) = r(t)/m$, $\beta(t) = \max\{ 1, \max_{0\leq s\leq t} \sqrt{m}\Vert a(t)\Vert_2\}$ and $\bar a(t) = a(t)/\beta(t)$ and assume that $\gamma_2^{(m)}$ is positive. For the step-sizes $\eta(t) = \beta(t) {\sqrt{2}}/(\Vert z\Vert_\infty\sqrt{t+1})$ and initialization $r(0)=0$, it holds
$$
\max_{0\leq s\leq t-1}\min_{i\in [n]} z_i^\top \bar a(s) \geq \gamma_2^{(m)} -  \frac{\Vert z\Vert_\infty}{\sqrt{t}}\left( 2\sqrt{2} +\frac{\sqrt{3}\log n}{\gamma_2^{(m)}} \right) .
$$
\end{proposition}
\begin{proof}
Using the fact that $a(t)=r(t)/m$ and $m \nabla F(r) = \nabla G(r/m)$ we have
$$
a(t+1) = a(t) +\frac{\eta(t)}{m}\nabla G(a(t)).
$$
It follows that 
$$
\frac{a(t+1)}{\beta(t)} = \bar a(t) +\frac{\eta(t)}{m\beta(t)} \nabla G_{\beta(t)}(\bar a(t)) =: b(t+1).
$$
and thus
$$
\bar a(t+1) = \frac{\beta(t)}{\beta(t+1)} \frac{a(t+1)}{\beta(t)} =  \frac{\beta(t)}{\beta(t+1)} b(t+1).
$$
Finally, since $\beta(t)\max\{1,\sqrt{m}\Vert b(t+1)\Vert_2\} = \max\{\beta(t),\sqrt{m}\Vert a(t+1)\Vert_2\}=\beta(t+1)$ it follows that $\bar a(t+1) = b(t+1)/\max\{1,\sqrt{m}\Vert b(t+1)\Vert_2\}$. Thus $\bar a(t)$ follows the iterations $\bar a(0)=0$ and
$$
\left\{
\begin{aligned}
b(t+1) &= \bar a(t) + \frac{\eta(t)}{m\beta(t)} \nabla G_{\beta(t)}(\bar a(t))\\
\bar a(t+1) &= b(t+1)/\max\{ 1, \sqrt{m}\Vert b(t+1)\Vert_2\}.
\end{aligned}
\right.
$$
These are \emph{online projected gradient ascent} iterations on the set $\{ a \in \RR^m \;;\; \sqrt{m}\Vert a\Vert_2\leq 1\}$ and for the sequence of functions $G_{\beta(t)}$, with step-size $\eta(t)/(m\beta(t))$.  Using the fact that the Lipschitz constant of $G_\beta$ if upper bounded by $\max_{i\in [n]} \Vert z_i\Vert_2 \leq \sqrt{m}\Vert z\Vert_\infty$ and the diameter of the constraint set is $2/\sqrt{m}$, we have the classical bound, with the step-size $\eta(t)/(m\beta(t)) = {\sqrt{2}}/(m\Vert z\Vert_\infty \sqrt{t+1})$,
$$
\frac1t \sum_{s=0}^{t-1} \left( G_{\beta(s)}(\bar a^*) - G_{\beta(s)}(\bar a(s)) \right) \leq \frac1t DL\sqrt{2t} =  \frac{2\sqrt{2}\Vert z\Vert_\infty}{\sqrt{t}}.
$$
Now using the bound of Eq.~\eqref{eq:comparison}, it follows
$$
\max_{0\leq s\leq t-1}\min_{i\in [n]} z_i^\top \bar a(s) \geq \gamma_2^{(m)} - \frac{\log n}{t}\sum_{s=0}^{t-1} \frac{1}{\beta(s)} - \frac{2\sqrt{2}\Vert z\Vert_\infty}{\sqrt{t}}.
$$
Let us now look at the evolution of $\beta(t)$. Let $\bar a^*$ be a max $\ell_2$-margin solution and remark that
$$
a(t+1)^\top \bar a^* - a(t)^\top \bar a^* = \eta(t) \nabla G(a(t))^\top \bar a^* \geq \eta(t) \gamma_2^{(m)}
$$
since $\nabla G(a(t)) \in \Delta^{m-1}$. It follows that $\beta(t) \geq \sqrt{m}\Vert a(t)\Vert_2 \geq m \gamma_2^{(m)} \sum_{s=0}^{s-1}\eta(s)$. 
Using the fact that $\beta(t) \geq C \sum_{s=0}^{s-1} \beta(s)/\sqrt{s+1}$ with $C= \gamma_2^{(m)} \sqrt{2}/\Vert z\Vert_\infty$ and the bound $\sum_{s=0}^{s-1} 1/\sqrt{s+1} \geq 2(\sqrt{t-1}-1)$, it follows that $\beta(t) \geq 2C \sqrt{t+1}/\sqrt{6}$ and thus 
$$ 
\sum_{s=0}^{t-1} \frac{1}{\beta(s)} \leq \frac{\sqrt{6}}{2C}\sum_{s=0}^{t-1} \frac{1}{\sqrt{s+1}} \leq\frac{\sqrt{6t}}{C}. 
$$
Plugging into the previous bound gives the conclusion. Note that we did not attempt to make the lower bound on $\beta(t)$ tight.
\end{proof}

\section{Appendix to Section~\ref{sec:generalization}: generalization bounds}\label{app:generalization}
With the notations of Section~\ref{sec:generalization}, let us first lower bound the margins in $\Ff_1$ and in $\Ff_2$.
\begin{lemma}
Assume that $\Vert x_i\Vert_2 \leq R$ for $i\in [n]$. For any $\epsilon \in (0,1)$ and $r\in [d]$, there exists $C(r),C_\epsilon(r)>0$ such that
\begin{align*}
\gamma_2 \geq \min\left\{ C(d), C_\epsilon (d) \left(\frac{\Delta_d(S_n)}{R}\right)^{\frac{d+3}{2-\epsilon}}\right\}
\quad\text{and}\quad
\gamma_1 \geq \min_{r \in [d]} \min \left\{C(r), C_\epsilon (r) \left(\frac{\Delta_r(S_n)}{R}\right)^{\frac{r+3}{2-\epsilon}}\right\}.
\end{align*}
\end{lemma}
\begin{proof}
Let $\dist_\mathcal{S}$ be the distance function to a set $\mathcal{S}$, i.e.\ $\dist_\mathcal{S}(x) = \inf_{\tilde x\in \mathcal{S}} \Vert x-\tilde x\Vert_2$, which is $1$-Lipschitz, and let $D_{\pm} = \{ x_i\;;\; y_i=\pm1\}$. For $P_r$ a projection that achieves the supremum in Eq.~\eqref{eq:interclass} (which exists by compactness of Grassmannians and continuity of the objective), we consider the following function
$$
f_r(x) = 2 \max\left( 0, 1-\frac{2 \dist_{P_r(D_+)}(P_r(x))}{\Delta_r(S_n)} \right) - 2\max\left( 0, 1-\frac{2 \dist_{P_r(D_-)}(P_r(x))}{\Delta_r(S_n)} \right).
$$
This function is $4/\Delta_r(S_n)$-Lipschitz continuous, satisfies $\Vert f\Vert_\infty \leq 2$ and $y_if(x_i)=2$ for all $i\in [n]$. Let us first consider the case $r=d$. Using the approximation results of Lipschitz functions in $\Ff_2$ from~\citet[Prop. 6]{bach2017breaking}, we know that if $N>0$ is larger than a constant independent of $\Delta_d(S_n)$ and satisfies
\[
C(d)\eta(N/\eta)^{-2/(d+1)}\log(N/\eta) \leq 1,
\]
where $\eta=\max\{2,4R/\Delta_d(S_n)\} =4R/\Delta_d(S_n)$, then there exists $\hat f$ such that $\Vert \hat f\Vert_{\Ff_2}\leq N$ and $\sup_{\Vert x\Vert_2 \leq R} \vert \hat f(x)-f_d(x)\vert \leq 1$. Since $\hat f/N$ is feasible for the $\Ff_2$-max-margin problem Eq.~\eqref{eq:maxmargin2}, this shows that $\gamma_2 \leq 1/N$ and it remains to estimate how large $N$ must be. In the next computations, the dimension dependent constant $C(d)$ might change from line to line. Using the bound $\log(u) \leq C(\epsilon,d)u^{\epsilon/(d+1)}$ for $\epsilon>0$, we obtain the stronger condition on $N$:
$$
C(\epsilon,d) \eta(N/\eta)^{(\epsilon-2)/(d+1)}\leq 1 \!\!\! \quad \Leftrightarrow\!\!\! \quad N \geq C(\epsilon,d)\eta^{(d+3-\epsilon)/(2-\epsilon)} \leq C(\epsilon,d) \left(\frac{\Delta_d(S_n)}{R}\right)^{(d+3)/(2-\epsilon)}\!\!\!\!.
$$
This gives the bound on $\gamma_2$.
For the bound on $\gamma_1$, it follows from the fact that for all $r\in [d]$, $\Ff_1$ contains the functions of the form $f\circ P_r$ where $f$ belongs to the space $\Ff_2$ over $\RR^r$ and $\Vert f\circ P_r\Vert_{\Ff_1}\leq \Vert f\Vert_{\Ff_2}$, see arguments and details in~\citet[Section 4.5]{bach2017breaking}.
\end{proof}

\begin{theorem}[Generalization bound]
For any $\epsilon \in (0,1)$ and $r\in [d]$, there exist $C(r),C_\epsilon(r)>0$ such that the following holds. If $(x,y)\sim \mathbb{P}$ is such that for some $R>0$ and $0<\Delta_r({\mathbb{P}})\leq C(r)$, it holds $\Delta_r(S_n) \leq \Delta_r(\mathbb{P})$ and $\Vert x\Vert_2\leq R$ almost surely, then for $f$ the $\Ff_1$-max-margin classifier, it holds with probability at least $1-\delta$ over the choice of i.i.d.~samples $S_n = (x_i,y_i)_{i=1}^n$,
\[
\mathbb{P}[yf(x)<0] \leq \frac{C_\epsilon(r)}{\sqrt{n}} \left(\frac{R}{\Delta_r(\mathbb{P})}\right)^{\frac{r+3}{2-\epsilon}} + \sqrt{\frac{\log(B)}{n}} + \sqrt{\frac{\log(1/\delta)}{2n}},
\]
where $B=\log_2(4(R+1)C_2(r))+(r+2)\log_2(R/\Delta_r(\mathbb{P}))$. The same bound holds for the $\Ff_2$-max-margin classifier for $r=d$.
\end{theorem}

\begin{proof}
This is a direct application of the margin-based generalization bounds of Theorem~\ref{thm:margingene}, using that for any $f\in \Ff_1$ with $\Vert f\Vert_{\Ff_1}\leq 1$, 
$$
\sup_{\Vert x\Vert_2\leq R} f(x) \leq \sup_{\Vert x\Vert_2\leq R, \theta \in \SS^{p-1}} \phi(\theta,x) \leq R+1,
$$
and the same holds in $\Ff_2$ since for $g\in \Ff_2$ it holds $g\in \Ff_1$ and $\Vert g\Vert_{\Ff_1}\leq \Vert g\Vert_{\Ff_2}$ by Jensen's inequality. We also use the Rademacher complexity bound 
$
\mathrm{Rad}_n(B_2) \leq \mathrm{Rad}_n(B_1) \leq \frac{1}{\sqrt{n}}.
$
where $B_i$ is the unit ball in $\mathcal{F}_i$. This can be found for instance in~\citet[Prop. 7]{bach2017breaking}.
\end{proof}

In the next theorem, $\Ff$ refers to a hypothesis class, $\mathrm{Rad}_n(\Ff)$ to its Rademacher complexity and $\gamma$ to its margin over the training set (see the cited reference for definitions).
\begin{theorem}[\citet{koltchinskii2002empirical}]\label{thm:margingene}  Assume that $\forall f\in \Ff$ we have $\sup_{x}\vert f(x)\vert \leq C$. Then, with probability at least $1-\delta$ over the sample, for all margins $\gamma>0$ and all $f\in \Ff$ we have
\[
\mathbb{P}[yf(x)<0] \leq 4\frac{\mathrm{Rad}_n(\Ff)}{\gamma} + \sqrt{\frac{\log(\log_2\frac{4C}{\gamma})}{n}} + \sqrt{\frac{\log(1/\delta)}{2n}}.
\]
\end{theorem}

\section{Proof for the ReLU case}\label{sec:relucase}
In this appendix, we detail how to rigorously cover the case of ReLU networks, i.e.~models as in~Eq.~\eqref{eq:infinitewidth} with a feature function of the form 
$$
\phi(w,x) = b\,(a^\top (x,1))_+,
$$
where $w=(a,b)\in \RR^{d+1}\times \RR$ (thus $w\in \RR^p$ with $p=d+2$) and $(u)_+=\max\{0,u\}$.  The difficulty with that case is that this function is not differentiable in $w$: (i) when $a=0$ for all $x$, or (ii) whenever $a^\top (x,1)=0$. We resolve these two sources of non-differentiability as follows: for (i), we consider a specific initialization that guarantees that $a$ does not vanish along the dynamics and for (ii), we consider an input distribution $\rho$ without atoms, in contrast to the empirical distribution $\rho = \frac1n \sum_{i=1}^n \delta_{x_i}$ that is implicitly used in the rest of the paper (as in~Eq.~\eqref{eq:smooth-margin}). 

\paragraph{Assumption on the input.}  We assume that the input distribution has the following properties:
\begin{enumerate}
\item[{\sf (A4)}] The input distribution $\rho \in \Pp(\RR^d)$ has a bounded density and a compact support denoted $\Xx$. The labels are given by a continuous deterministic function $y:\Xx\to \Yy=\{-1,1\}$.
\end{enumerate}
The assumption on $\rho$ excludes discrete measures. Also, the continuity assumption on $y$ means that the sets where $y=+1$ and where $y=-1$ are disconnected, and this implies that they are at a positive distance from each other since the level-sets $\{y=1\}$ and $\{y=-1\}$ are both compact and with empty intersection. This distribution $\rho$ could be for instance the population distribution of the input, or also could be obtained by taking the expectation over small perturbations of the input training set, which is a well-known smoothing technique (e.g.,~\citet{duchi2012randomized}).  This leads to the definition of the \emph{population smooth-margin}
$$
S(f) = -\log \Big(\int_{\Xx} \exp(-f(x))\d\rho(x)\Big)
$$
defined for $f\in \Cc(\Xx)$, where $\Cc(\Xx)$ is the space of continuous and real-valued functions on $\Xx$ endowed with the supremum norm. Let us give some facts about this function.

\begin{lemma}\label{lem:reluloss}
The function $S:\Cc(\Xx) \to \RR$ is Fr\'echet differentiable with a gradient at $f\in \Cc(\Xx)$ given by $\nabla S[f]\in \Pp(\Xx)$:
$$
\nabla S[f](\!\d x) = \frac{\exp(-f(x))\rho(\!\d x)}{\int_{\Xx} \exp(-f(x'))\d\rho(x')}.
$$
Also, as a function $\Cc(\Xx)\to L^1(\rho)$, the function $f \mapsto \frac{\d \nabla S[f]}{\d \rho}$ is Lipschitz continuous on bounded sets in $\Cc(\Xx)$.
\end{lemma}
\begin{proof}
Let $\ell=\exp$. Since $\ell'$ is locally Lipschitz continuous, for any $f,\tilde f\in \Cc(\Xx)$ there exists $L>0$ such that $\vert \ell(-(f(x)+\epsilon \tilde f(x))) - \ell(-f(x)) + \epsilon \ell'(-f(x))\tilde f(x) \vert \leq (L^2/2)\epsilon^2\Vert \tilde f\Vert_{\infty}$. Thus when $\epsilon \to 0$, we have $\sup_{x\in \Xx}\vert \epsilon^{-1}(  \ell(-(f(x)+\epsilon \tilde f(x))) - \ell(-f(x))) +\ell'(-f(x))\tilde f(x)\vert\to 0$ which shows that $f\mapsto \int_\Xx \exp(-f)\d\rho$ is Fr\'echet differentiable with gradient $-\exp(-f)\d\rho \in \Mm(\Xx)$. The differentiability of $S$ follows by composition with $-\log$. The Lipschitz continuity of $\d\nabla S/\d\rho$ on bounded sets is a consequence of the Lipschitz continuity of $\exp$ and the existence of positive lower bounds for $\exp$ on bounded intervals.
\end{proof}

\begin{lemma}[Convergence of soft-argmin]\label{lem:argminrelu}
Let $\rho\in \Pp(\RR^d)$ have a compact support $\Xx$. Let $\beta_t >0$ and $f_t\in \Cc(\Xx)$ be sequences such that $\beta_t \to \infty$ and $f_t  \to f_\infty \in \Cc(\Xx)$ (uniformly) as $t\to \infty$. If $\lambda_t = \nabla S [\beta_t f_t]$ converges weakly to some $\lambda \in \Pp(\Xx)$, then $\spt(\lambda)\subset  \arg\min_{x\in \Xx} f_\infty(x)$.
\end{lemma}
\begin{proof}
Let $x^*, x_0 \in \Xx$ be such that $x^*\in \arg\min_{x\in \Xx} f_\infty$ and $x_0 \notin \arg\min f_\infty$. By continuity of $f_\infty$ and uniform convergence, there exists $\epsilon, \delta, t_0>0$ such that $\forall t\geq t_0$, it holds $f_t(x)\geq f(x^*)+\epsilon$ for all $x\in B_\delta(x_0)$ and $f_t(x)\leq f(x^*)+\epsilon/2$ for all $x\in B_\delta(x^*)$, where $B_\delta(x)$ is the open ball of radius $\delta$ centered at $x$. It follows
$$
\lambda_t(B_\delta(x_0)) = \frac{\int_{B_\delta(x_0)}\exp(-\beta_t f_t(x))\d\rho(x) }{\int_\Xx \exp(-\beta_t f_t(x)) \d\rho(x)} \leq \frac{\exp(-\beta (f(x^*)+\epsilon))\rho(B_\delta(x_0))}{\exp(-\beta (f(x^*)+\epsilon/2))\rho(B_\delta(x^*))} \to 0.
$$
This shows that $x_0\notin \spt(\lambda)$. Since this is true for any $x_0\notin \arg\min f_\infty$, it follows that $\spt(\lambda)\subset \arg\min_{x\in \Xx} f_\infty(x)$.
\end{proof}

Now, we define the objective
$$
F(\mu) = S(\hat h(\mu))
$$
where $\hat h(\mu) \in \Cc(\Xx)$ is the function $x\mapsto y(x)h(\mu,x)$. Seing $F$ as a function on unnormalized measures, it admits a Fr\'echet differential at $\mu\in \Pp_2(\RR^p)$, represented by $F_\mu'\in \Cc(\RR^p)$ given by
$$
F'_\mu(w) = \int_{\Xx} y(x)\phi(w,x) \d(\nabla S[\hat h(\mu)])(x).
$$
Our next step is to gather some regularity properties of $F'_\mu$.
Let us consider the sets 
\begin{align}\label{eq:sets_relu}
D := \{(a,b)\in \RR^{p-1}\times \RR\;;\; \Vert a\Vert =\vert b\vert\}&&\text{and}&& \SS_\pm := \SS_{+}\cup \SS_{-} := \SS^{p-1}\cap D.
\end{align}
 We endow $\SS_{\pm}$ with its Riemannian geometry inherited from the sphere (it is a disconnected manifold with two connected components $\SS_+$ and $\SS_-$) and let us denote $J'_\mu$ the restriction of $F'_\mu$ to $\SS_{\pm}$. 
\begin{lemma}\label{lem:differentiability_relu}
Under Assumption~{\sf (A4)} for all $\mu \in \Pp_2(\RR^p)$, the function $F'_\mu$ is differentiable on $\{(a,b)\in \RR^p\;;\; a\neq 0\}$. Also let $A\subset \Pp_2(\RR^p)$ be a set of measures with uniformly bounded second moments. Then  $\nabla J_\mu'$ is Lipschitz continuous on $\SS_{\pm}$, uniformly on $A$ and the function $J':\Pp_2(\RR^p)\to \Cc^1(\SS_{\pm})$ is Lipschitz continuous on $A$.
\end{lemma}
\begin{proof}
Let us denote $\lambda_\mu = \nabla S[\hat h(\mu)]$, let  $w,\bar w\in \RR^{p}$ where $w=(a,b)$ with $a\neq 0$ and $V\subset \RR^{p}$ be a convex bounded open set that contains $w$ and $\bar w$. Since the function $w\mapsto y(x)\phi(w,x)$ is Lipschitz on $V$, uniformly in $x\in \Xx$, there exists a constant $L>0$ such that $\epsilon^{-1}\vert \phi(w+\epsilon\bar w,x)-\phi(w,\epsilon)\vert \leq L$ for $\epsilon$ small enough. Moreover, if $x\notin H_a:= \{x\in \Xx\;;\; a^\top x=0\}$, then $\epsilon^{-1}(\phi(w+\epsilon\bar w,x)-\phi(w,\epsilon)) \to  \nabla_w \phi(w,x)^\top \bar w$ as $\epsilon\to 0$. Since $\rho(H_a)=0$ it follows by the dominated convergence theorem that
$$
\lim_{\epsilon\to 0} \int_\Xx \vert \epsilon^{-1}y(x)(\phi(w+\epsilon\bar w,x)-\phi(w,\epsilon)) - y(x)\nabla_w \phi(w,x)^\top \bar w\vert \d\lambda_\mu(x) =0,
$$
which shows that $F'_\mu$ is differentiable with the gradient $w\mapsto \int_\Xx y(x)\nabla_w \phi(w,x)\d\lambda_\mu(x)$ if $a\neq 0$.
 
Let us now show that $\nabla F'_\mu$ is Lipschitz continuous on $\SS_{\pm}$, which implies the same for $\nabla J'_\mu$. Since this is immediate for the component $\nabla_b$ so let us focus on  $\nabla_a$. 
For $(a,b),(\bar a,b) \in \SS_+$, defining $H_{a,\bar a} = \{ x\in \Xx\;;\; (a^\top x)_+^0 \neq (\bar a^\top x)_+^0\}$, we have
$$\Vert \nabla_a F'_{\mu}(a,b) - \nabla_a F'_{\mu}(\bar a,b)\Vert \leq 2^{-1/2} \sup_{x\in \Xx}\Vert (x,1)\Vert\lambda_\mu(H_{a,\bar a}).$$
But since the Lebesgue measure of $ \{x \in \Xx \;;\; (x,1)\in H_{a,\bar a}\}$ is bounded by a constant times $\dist(a,\bar a)$ and $\lambda_\mu$ has a bounded density with respect to the Lebesgue measure, it follows that for some constant $L$, $\lambda_\mu(H_{a,\bar a}) \leq L\dist(a,\bar a)$. Moreover this constant $L$ is uniform over measures $\lambda_\mu$ with a bounded density with respect to $\rho$.  But if $A\subset \Pp_2(\RR^p)$ is such that for some $C>0$, $\int \Vert w\Vert_2^2\d\mu <C$ for all $\mu\in A$, then $\{ \lambda_\mu \; ;\; \mu \in A\}\subset \Pp(\Xx)$ is a set of measures with uniformly bounded densities with respect to $\rho$, so $L$ is uniform over $A$.

For the last claim, notice that $\hat h:\Pp_2(\RR^p)\to \Cc(\Xx)$ is Lipschitz continuous and that $(\d\nabla S/\d\rho):\Cc(\Xx)\to L^1(\rho)$ is Lipschitz continuous, by Lemma~\ref{lem:reluloss}. It follows that $\mu\mapsto (\d \lambda_\mu/\d\rho)$ is Lipschitz continuous as a function $\Pp_2(\RR^p)\to L^1(\rho)$. 
Finally, $M:= \sup_{w\in \SS_\pm} \sup_{x\in \Xx} \{\vert \phi(w,x)\vert, \Vert \nabla \phi(w,x)\Vert\}$ is finite (here $\sup_{x\in \Xx}$ is an essential supremum w.r.t.~$\rho$), so it follows that $J':\Pp_2(\RR^p)\to \Cc^1(\SS_{\pm})$ is Lipschitz continuous because $\Vert J'_\mu - J'_{\tilde \mu}\Vert_{\Cc^1} \leq M \Vert (\d \lambda_\mu/\d\rho) - (\d \lambda_{\tilde \mu}/\d\rho) \Vert_{L^1(\rho)}$.
\end{proof}

\paragraph{Existence of a Wasserstein gradient flow.} After these preliminaries, we are in position to show the existence of Wasserstein gradient flows for certain initializations.  In this section, we will not attempt to prove uniqueness of the dynamics.
\begin{lemma}\label{lem:existence_relu}
Under Assumption~{\sf (A4)}, let $\mu_0\in \Pp_2(\RR^{d+2})$ satisfying $\spt(\mu_0)\subset D$. Then there exists a well-defined Wasserstein gradient flow of $F$ starting from $\mu_0$. It satisfies $\spt(\mu_t)\subset D$ for $t\geq 0$.
\end{lemma}

\begin{proof}
We will essentially follow the proof of~\cite[Thm. 2.6]{chizat2018global}, by highlighting the points that need to be adapted. Without loss of generality (the trivial case $\mu_0 =\delta_{0}$ put aside), we assume that $\spt(\mu_0)$ is included in a compact set $K\subset D$ that does not contain $0$, since the general case can be treated as in Theorem~\ref{th:WGF} by exploiting the $2$-homogeneity of $F'_{\mu}$ for all $\mu \in \Pp_2(\RR^p)$.

Let $\mu_{0,m}$ be a sequence of empirical measures concentrated in $K$ that converges weakly to $\mu_0$. For all $m$, the Wasserstein gradient flow equation Eq.~\eqref{eq:flowdef} boils down to an ordinary differential equation which existence is guaranteed by Peano existence theorem and Lemma~\ref{lem:differentiability_relu}, at least on some maximal interval $[0,t_0[$. If $t_0<\infty$, the growth of $\nabla F'_{\mu_{m,t}}$ for $t\in [0,t_0[$ is such that Gr\"onwall type arguments guarantee that the atoms of $\mu_{m,t}$ are uniformly bounded (also bounded away from $0$) and converge as $t\to t_0$, thus defining a measure $\mu_{m,t_0}$. We denote $X(t,w) = (A_t(w),B_t(w))\in \RR^{p-1}\times \RR$ where $X$ is the flow from Eq.~\eqref{eq:flowdef}. It holds, denoting $\tilde x = (x,1)$ and expanding Eq.~\eqref{eq:flowdef}:
\begin{align*}
\frac{\d}{\d t} \frac12 \Vert A_t(w)\Vert^2 &= A_t(w)^\top \int B_t(w) (A_t(w)^\top \tilde x)_+^0 \tilde x \d[\nabla S(\hat h(\mu_{t,m}))](x),\\
\frac{\d}{\d t} \frac12 \vert B_t(w)\vert^2 &= B_t(w) \int  (A_t(w)^\top \tilde x)_+ \d[\nabla S(\hat h(\mu_{t,m}))](x).
\end{align*}

\begin{wrapfigure}{r}{0.52\textwidth}
  \vspace{-0.8cm}
  \begin{center}
    \includegraphics[width=0.53\textwidth]{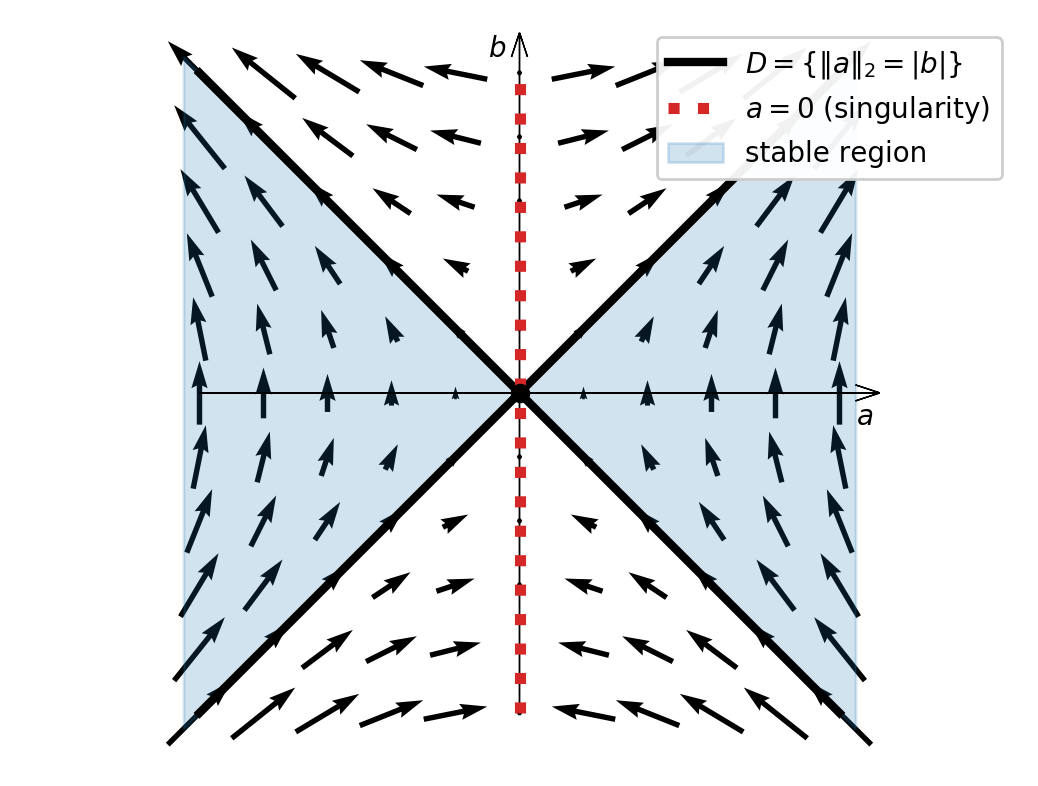}
  \end{center}
  \vspace{-1.1cm}
  \caption{ $2$-d slice of the vector field $\nabla F'_{\mu_t}$ in the ReLU case (illustration). Initializing on the invariant set $D$ allows to avoid $a=0$. }\label{fig:invariantset}
  \vspace{-0.5cm}
\end{wrapfigure}

These quantities being equal, this shows that $\forall w\in \RR^p$, $\Vert A_t(w)\Vert_2^2 - \vert b_t(w)\vert^2$ is constant on $[0,t_0[$, as long as $A_t(w)\neq 0$ on this interval. It follows that whenever $w \in D$, we have $\Vert A_t(w)\Vert =\vert B_t(w)\vert$ for all $t\in [0,t_0[$. Since $\mu_{t,m} = (X_t)_\# \mu_{0,m}$, we have that if $\spt(\mu_{m,0})\subset D$ then $\spt(\mu_{m,t}) \subset D$ for all $t\in [0,t_0[$ and since $D$ is closed, we have $\spt(\mu_{t_0,m})\subset D$ which contradicts the maximality of $t_0$ (the trajectory could be extended by Peano existence theorem). Thus $t_0=\infty$. The principle behind the initialization on $D$ is illustrated on Figure~\ref{fig:invariantset} and is related to well-known properties of homogeneous models~\citep{du2018algorithmic}.

So far, we have proved that for all $m$, there exists a Wasserstein gradient flow $\mu_{t,m}$ starting from $\mu_{0,m}$, defined on $[0,\infty[$ and such that $\spt(\mu_{t,m})\subset D$ for $t\geq 0$. The rest of the proof follows that in~\cite[Thm. 2.6]{chizat2018global} where we extract a (weak) limit curve $\mu_{t}$ by compactness and show that it satisfies the Wasserstein gradient flow equation Eq.~\eqref{eq:flowdef}. The only technical point is Step.(iii) in that proof, which is taken care of by the last claim of Lemma~\ref{lem:differentiability_relu}. Note that here we do not prove uniqueness of the Wasserstein gradient flow, only its existence.
\end{proof}

\paragraph{Implicit bias for ReLU networks.}
Let us restate Theorem~\ref{th:bias_relu} in a slightly more general form and making Assumption~{(*)} explicit. We recall that $\nu_t := \Pi_2(\mu_t)$ and $\bar \nu_t := \nu_t/(\nu_t(\SS^{p-1}))$.

\begin{theorem}
Let $\mu_0 \in \Pp_2(\RR^p)$ be such that $\spt(\Pi_2(\mu_0)) =  \SS_{\pm}$ and assume~{\sf (A4)}. Then there exists $(\mu_t)_{t\geq 0}$ a Wasserstein gradient flow of the objective Eq.~\eqref{eq:objectiverelu} with initialization $\mu_0$. Moreover, 
\begin{itemize}
\item if $\bar \nu_t = \Pi_2(\mu_t)/([\Pi_2(\mu_t)](\SS^{p-1}))$ converges weakly to some $\bar \nu_\infty \in \Pp(\SS^{p-1})$,
\item if $\nabla S[\hat h(\mu_t)] $ converges weakly to some $\lambda_\infty  \in \Pp(\Xx)$, and
\item (*) if $J'_{\mu_t}$ converges in $\mathcal{C}^1(\SS_{\pm})$ to some $J'_\infty$ that satisfies the Morse-Sard property: the set of critical values of $J'_\infty$ (i.e.,~$v\in \RR$ such that there exists $\theta \in \SS_{\pm}$ such that $J'_\infty(\theta)=v$ and $\nabla J'_\infty(\theta)=0$) has Lebesgue measure zero,
\end{itemize}
then $h(\bar \nu_{\infty},\cdot)$ is a maximizer for the $\mathcal{F}_1$-max-margin problem 
$
\max_{\Vert f\Vert_{\Ff_1}\leq 1} \min_{x\in \spt(\rho)} y(x)f(x).
$
\end{theorem}

Before proving this theorem, let us discuss its assumptions.  First, the assumption on the initialization given here is satisfied by $\mu_0$ given in Theorem~\ref{th:bias_relu} (which is an example given for the sake of concreteness). The conditions in the two first bullets are similar to those of Theorem~\ref{th:biasWGF} and just require the uniqueness of limits of some sequences that live in compact spaces. In Assumption (*), the fact that the Morse-Sard property holds for $J'_\infty$ was already required in~\cite{chizat2018global} and it is an open question to guarantee that this property holds  in this context (where $J'_\infty$ is potentially an infinite sum of subanalytic functions, instead of a finite sum as in Theorem~\ref{th:bias_relu}). Finally, the most undesirable assumption is perhaps the convergence of $J'_{\mu_t}$ to $J'_\infty$ in $\Cc^1(\SS_{\pm})$. The fact that it converges in $\Cc(\SS_\pm)$ can be shown a priori, so the assumption is really on the uniform convergence of the gradient. In particular, it requires $J'_\infty$ to be continuously differentiable, which is for instance not true if $\lambda_\infty$ is a discrete measure.

\begin{proof}
Since $\spt(\mu_0) \subset D$, the existence of a Wasserstein gradient flow is proved in Lemma~\ref{lem:existence_relu}. Let $\lambda_t =\nabla S[h(\mu_t)] \in \Pp(\Xx)$ which, as assumed, converges weakly to $\lambda_\infty$.  Since $y(x)\lambda_\infty(\d x)$, seen as a signed measure on $\Xx$ is non-zero, it follows by the contraposition of Lemma~\ref{lem:relu_injectivity} that $J'_\infty$ is not identically $0$ on $\SS_+$. Moreover, since $J'_\infty$ takes opposite values on $\SS_+$ and $\SS_-$, we have that $M\coloneqq \max_{w \in \SS_\pm} J'_\infty(w) >0$. For the rest of the proof, we just comment on how to adapt the arguments of Theorem~\ref{th:biasWGF} to this case.

\textbf{Step 1.} The argument of Theorem~\ref{th:biasWGF} goes through if we replace $\SS^{p-1}$ by the set $\SS_{\pm}$, in particular thanks to Assumption (*). Lemma~\ref{lem:differentiability_relu} and the positive $2$-homogeneity guarantee that the restriction of the flow $X$ to $D$ is a diffeomorphism.

\textbf{Step 2.}  Since we focus on the exponential loss, we can directly apply Lemma~\ref{lem:argminrelu}, which gives $\spt(\lambda_\infty) \subset \arg\min_{x\in \spt(\rho)} y(x)h(\bar \nu_\infty, x)$. 

\textbf{Steps 3.} The argument of Theorem~\ref{th:biasWGF} goes through, again replacing $\SS^{p-1}$ by the set $\SS_{\pm}$.

\textbf{Steps 4.} We conclude with the optimality conditions in Proposition~\ref{prop:optimalityrelu} and the structure of the minimizers given in Proposition~\ref{prop:equivnorms}.
\end{proof}

\begin{lemma}\label{lem:relu_injectivity}
Assume that $\Xx\subset \RR^d$ is compact. If $\lambda \in \Mm(\Xx)$ satisfies $\int (a^\top (x,1))_+ \d\lambda(x)=0$ for all $a\in \SS^{d}$, then $\lambda = 0$.
\end{lemma}
\begin{proof}
Consider the measure $\tilde \lambda\in \Mm(\SS^{d})$ that is such that $\int \varphi((x,1))\d\lambda(x) = \int \varphi(z)\d\tilde \lambda(z)$ for all positively homogeneous functions $\varphi \in \Cc(\RR^{d+1})$. By construction, $\tilde \lambda$ is concentrated on the set $S_\alpha$ of points in $\SS^{d}$ for which the last coordinate is larger than $\alpha$ for some $\alpha>0$.  Let us show that $\tilde \lambda=0$, which implies that $\lambda =0$ and thus the claim.

Let $\d\tau \in \Pp(\SS^{d})$ be the uniform distribution on the sphere. It is known~\citep{bach2017breaking} that the set of functions $\{f_g: a\mapsto \int (a^\top z)_+g(z)\d\tau(z);\;;\; g\in L^2(\tau)\}$ is dense in $\Cc(S_\alpha)$.  Since for $g\in L^2(\tau)$ it holds, by Fubini,
$$
\int f_g\d\tilde \lambda = \int \d\tilde \lambda(a) \int \d\tau(z) (a^\top z)_+ g(z)  = \int \d\tau(z) \int \d\tilde \lambda(a) (a^\top z)_+ g(z)=0,
$$
it follows that $\tilde \lambda=0$ and thus $\lambda=0$.
\end{proof}
Finally, let us state the optimality conditions of the optimization problem mentionned in Theorem~\ref{th:bias_relu} (it is an application of minimax duality~\citep{sion1958general}, just like Proposition~\ref{prop:optimalityrelu}). For $\Xx\subset \RR^d$ compact, let
\begin{equation}\label{eq:maxmargin1_relu}
\gamma_1 \coloneqq \max_{\Vert f\Vert_{\Ff_1}\leq 1}\min_{x\in \Xx}\  y(x) f(x) = \max_{\substack{\nu \in \Pp(\SS^{p-1})}}\min_{\lambda \in \Pp(\Xx)}\  \int_{\Xx} \int_{\SS^{p-1}}  y(x)\phi(\theta,x)\d\nu(\theta)\d\lambda(x).
\end{equation}
\begin{proposition}[Optimality conditions - ReLu]\label{prop:optimalityrelu}
The maximization problem~\eqref{eq:maxmargin1_relu} admits global maximizers $\nu^\star \in \Pp(\SS^{p-1})$. Moreover, a measure $\nu^\star \in \Pp(\SS^{p-1})$ is a global maximizer of~\eqref{eq:maxmargin1} if and only if $\nu^\star \in \Pp(\SS^{p-1})$ and there exists $\lambda^\star \in \Pp(\Xx)$ such that 
\begin{align*}
\spt \nu^\star \subset \arg\max_{\theta\in \SS^{p-1}} \int_{\Xx} y(x) \phi(\theta,x)\d\lambda^*(x)
\qandq
\spt \lambda^\star \subset \arg\min_{x\in \Xx}  \int_{\SS^{p-1}} y(x)  \phi(\theta,x)\d\nu^\star(\theta).
\end{align*}
\end{proposition}

\end{document}